\documentclass[12pt]{amsart} 
\usepackage{amssymb}
\usepackage{latexsym}
\usepackage{amsmath}
\usepackage{amsthm}
\usepackage{amsfonts}

\usepackage{mathabx}

\newtheorem{theorem}{Theorem}

\newtheorem{lemma}[theorem]{Lemma}
\newtheorem{corollary}[theorem]{Corollary}
\newtheorem{proposition}[theorem]{Proposition}

\theoremstyle{definition}

\newtheorem*{acknowledgements*}{Acknowledgements}
\theoremstyle{remark}
\newtheorem{remark}[theorem]{Remark}

\numberwithin{equation}{section}
\numberwithin{theorem}{section}

\newcommand{\R}{\mathbb R}
\newcommand{\N}{\mathbb N}
\newcommand{\C}{\mathbb C}

\title
[The finite Hilbert transform]
{Inversion and extension of the finite \\ 
Hilbert transform on $\mathbf{(-1,1)}$}

\author[G.P.  Curbera]{Guillermo P. Curbera}
\address{Facultad de Matem\'aticas \& IMUS,
Universidad de Sevilla, 
Calle Tarfia s/n,  Sevilla 41012, Spain}
\email{curbera@us.es}

\author[S. Okada]{Susumu Okada}
\address{School of Mathematics and Physics, University of Tasmania, 
Private Bag 37, Hobart, Tas. 7001, Australia}
\email{susbobby@grapevine.com.au}

\author[W.J. Ricker]{Werner J. Ricker}
\address{Math.--Geogr. Fakult\"at, Katholische Universit\"at
Eichst\"att--Ingolstadt, D--85072 Eichst\"att, Germany}
\email{werner.ricker@ku.de}

\thanks{The first author acknowledges the support  of 
MTM2015-65888-C4-1-P, MINECO (Spain).}
\thanks{The second author was supported by the 
``International Visiting Professor Program 2017'', via the Ministry
of Education, Science and Art, Bavaria (Germany).}

\date{\today}



\begin{document}

\begin{abstract}
The principle of optimizing inequalities, or their equivalent 
operator theoretic formulation, is well established in analysis.
For an operator, this corresponds to extending its action 
to larger domains, hopefully to the largest possible 
such domain (i.e, its \textit{optimal domain}).
Some classical operators are already  optimally defined (e.g.,
the Hilbert transform in $L^p(\mathbb{R})$, $1<p<\infty$) and others are not
 (e.g., the Hausdorff-Young inequality in $L^p(\mathbb{T})$, $1<p<2$, or Sobolev's
inequality in various spaces). In this paper a detailed investigation
is undertaken of the finite Hilbert transform $T$ acting on rearrangement 
invariant spaces $X$ on  $(-1,1)$, an operator whose 
singular kernel is neither positive nor does it possess any monotonicity properties.
For a large class of such spaces $X$  it is shown that  
$T$ is already  optimally defined on $X$ 
(this is known for $L^p(-1,1)$ for all
$1<p<\infty$, except $p=2$).
The case $p=2$ is significantly different because the range of $T$ 
is a proper dense subspace of $L^2(-1,1)$. 
Nevertheless, by a completely different approach, 
it is established that $T$ is also  optimally defined on
$L^2(-1,1)$. Our methods are also used to show that the solution 
of the airfoil equation, 
which is well known for the spaces $L^p(-1,1)$ whenever $p\not=2$
(due to certain properties of $T$),
can also be extended to the class of r.i.\ spaces $X$ considered in this paper.
\end{abstract}

\maketitle


\section{Introduction}


For $1\le p\le2$ the Fourier transform $F$ maps 
$L^p(\mathbb{T})$ into $\ell^{p'}(\mathbb{Z})$,
with $\frac{1}{p}+\frac{1}{p'}=1$. The Hausdorff-Young inequality $\|F(f)\|_{p'}\le \|f\|_p$
for $f\in L^p(\mathbb{T})$ ensures that $F$ is continuous. 
The following question was raised by
R.\ E.\ Edwards, \cite[p. 206]{edwards}, 50 years ago: \textit{Given $1\le p\le 2$, 
what can be said about the space $\mathbf{F}^p(\mathbb{T})$ consisting of those  
functions $f\in L^1(\mathbb{T})$ having the property that 
$F(f\chi_A)\in \ell^{p'}(\mathbb{Z})$ for all sets $A$ in 
the Borel $\sigma$-algebra $\mathcal{B}_{\mathbb{T}}$ on $\mathbb{T}$?}
A consideration of the  functional 
\begin{equation}\label{fourier}
f\mapsto\sup_{A\in\mathcal{B}_{\mathbb{T}}}\left\|F(\chi_Af)\right\|_{p'},
\end{equation}
would be expected to be relevant in this regard. 
For $p=2$, the operator 
$F\colon L^2(\mathbb{T})\to \ell^2(\mathbb{Z})$ 
is a Banach space \textit{isomorphism}, 
which implies that $\mathbf{F}^2(\mathbb{T})=L^2(\mathbb{T})$.
What about the case $1<p<2$?
It turns out that the  functional \eqref{fourier} is a norm, that 
$\mathbf{F}^p(\mathbb{T})\subseteq L^1(\mathbb{T})$ is a Banach function 
space (briefly, B.f.s.) 
\textit{properly} containing $L^p(\mathbb{T})$, and that
$F\colon \mathbf{F}^p(\mathbb{T})\to \ell^{p'}(\mathbb{Z})$
is continuous. 
Moreover, $\mathbf{F}^p(\mathbb{T})$
is the \textit{largest} such space in a certain sense. For the above
facts we refer to \cite{mockenhaupt-ricker}.
The point is that
the Hausdorff-Young inequality for functions in $L^p(\mathbb{T})$, $1<p<2$,
\textit{can} be extended to its
\textit{genuinely larger}  optimal domain space $\mathbf{F}^p(\mathbb{T})$.

For many classical inequalities in analysis, or their
equivalent operator theoretic formulation, an investigation along the lines of the
Hausdorff-Young inequality alluded to above can be quite fruitful.
One has a linear operator $S$ defined on some  B.f.s.\   $Z\subseteq L^0(\mu)$, 
with $(\Omega,\Sigma,\mu)$ a measure space, taking values in a Banach space $Y$ and
a B.f.s.\ $X\subseteq Z$ such that $S\colon X\to Y$ is bounded. 
The above question posed by Edwards is also meaningful in this setting: 
What can be said about the space $X_S$ consisting of those functions
$f\in Z$ satisfying $S(f\chi_A)\in Y$ for all $A\in\Sigma$?
In particular, is $X_S$ genuinely  larger than $X$? If so,
can $X_S$ be equipped with a function norm such that
$X\subseteq X_S$ continuously and $S$ has a
$Y$-valued, continuous linear extension to $X_S$?
And, of course, $X_S$ should be the \textit{largest}
space with these properties. A few examples will illuminate this discussion.

Let $\Omega\subset\mathbb{R}^n$ be a bounded domain with $|\Omega|=1$.  
The validity of the generalized Sobolev inequality $\|u^*\|_Y\le C \||\nabla u|^*\|_X$
for $u\in C_0^1(\Omega)$, where $v^*$ is the decreasing rearrangement
of a function $v$ and $X,Y$ are rearrangement invariant 
(briefly, r.i.) spaces on $[0,1]$, is equivalent to the 
boundedness of the inclusion operator $j\colon W_0^1X(\Omega)\to Y(\Omega)$
for a suitable Sobolev space $W_0^1X(\Omega)$. By using a 
generalized Poincar\'e inequality, Cwikel and Pustylnik, \cite{cwikel-pustylnik}, and
Edmunds, 
Kerman and Pick, \cite{edmunds-kerman-pick}, showed that the boundedness of $j$ 
is equivalent to the boundedness, from $X$ into $Y$, of the 1-dimensional
operator $S$ associated with Sobolev's inequality, namely,
$$
(S(f))(t):=\int_t^1f(s)s^{(1/n)-1}ds,\quad t\in[0,1],
$$
which is generated by the kernel $K(t,s):=s^{(1/n)-1}\chi_{[t,1]}$ on $[0,1]\times[0,1]$.
Accordingly, being able to extend the operator $S$ is equivalent to extending the
imbedding $j$ and hence, to refining the generalized Sobolev inequality.
The  optimal extension of this kernel operator $S$ is treated in
\cite{curbera-ricker-sm};
whether or not the initial space becomes genuinely larger 
depends on properties of $X$ and $Y$. A knowledge of the
optimal domain of $S$ has implications for the compactness of the Sobolev 
imbedding $j$, \cite{curbera-ricker-tams}, \cite{curbera-ricker-ind}.

For $0<\alpha<1$, the classical fractional integral operator  in the spaces $L^p(0,1)$, $1\le p\le \infty$,
has kernel (up to a constant) given by $K(t,s)=|s-t|^{\alpha-1}$. 
Its  optimal extension  has been investigated in 
\cite{curbera-ricker-nach}. For convolution
(and more general Fourier multipliers) operators  in $L^p(G)$, $1\le p<\infty$, with $G$ a 
compact abelian group, see \cite{mockenhaupt-okada-ricker}, 
\cite[Ch.7]{okada-ricker-sanchez} and the references therein.
The  optimal extension of the classical Hardy operator 
in $L^p(\mathbb{R})$, $1<p<\infty$, with kernel $K(t,s):=(1/t)\chi_{[0,t]}(s)$
has been investigated in \cite{delgado-soria}.

In this paper we consider another classical singular 
integral operator. The Hilbert transform 
$H\colon L^p(\R)\to L^p(\R)$, for $1<p<\infty$
(whose boundedness is due to M.\ Riesz), is defined via 
convolution as a principal value integral; see, for example, \cite[\S6.7]{edwards-gaudry}. 
Since $H^2=-I$, the operator $H$ is a Banach space \textit{isomorphism} 
on $L^p(\R)$ for every $1<p<\infty$ and so there is \textit{no} larger B.f.s.\ 
which contains $L^p(\R)$ and such that $H$ 
has an $L^p(\R)$-valued extension to this space. 
A related operator is the 
Hilbert transform $H_{2\pi}$ of $2\pi$-periodic functions 
defined via the principal value integrals
\begin{equation*}\label{FHT-2pi}
(H_{2\pi}(f))(x)=p.v. \frac{1}{2\pi}\int_{-\pi}^{\pi}f(x-u)\cot(u/2)\,du
\end{equation*}
for every measurable $2\pi$-periodic function $f$ and for every point $x\in[-\pi,\pi]$ for 
which the p.v.-integral exists.
For each $1<p<\infty$, the operator $H_{2\pi}$ is  linear 
and continuous from $L^p(-\pi,\pi)$ into itself; denote this operator by
$H^p_{2\pi}$. It is known that $H^p_{2\pi}$ 
has proper closed range,  \cite[Sect. 9.1]{butzer-nessel}.
Hence, $H^p_{2\pi}$ is surely not an isomorphism on 
 $L^p(-\pi,\pi)$. Nevertheless,  
as for $H$, it turns out that there is \textit{no} genuinely larger B.f.s.\ 
containing $L^p(-\pi,\pi)$ 
such that $H^p_{2\pi}$ has an $L^p(-\pi,\pi)$-valued extension to this space, 
\cite[Example 4.20]{okada-ricker-sanchez}.

The finite Hilbert transform $T(f)$ 
of $f\in L^1(-1,1)$ is the principal value integral
\begin{equation*}
(T(f))(t)=\lim_{\varepsilon\to0^+} \frac{1}{\pi}
\left(\int_{-1}^{t-\varepsilon}+\int_{t+\varepsilon}^1\right) \frac{f(x)}{x-t}\,dx ,
\end{equation*}
which exists for a.e.\ $t\in(-1,1)$ and is a measurable function. 
It  is  known to have important applications to aerodynamics, 
via the resolution of the so-called  \textit{airfoil  equation}, 
\cite{cheng-rott}, \cite[Ch.11]{king}, \cite{reissner},  \cite{tricomi-1},
\cite{tricomi}. More recently, the finite Hilbert transform has also found applications 
to problems arising in image reconstruction; 
see, for example, \cite{katsevich-tovbis}, \cite{sidky-etal}.
For each $1<p<\infty$
the linear operator $f\mapsto T(f)$ maps $L^p(-1,1)$ continuously  into 
itself (denote this operator by $T_p$).
Except when $p=2$, the operator $T_p$ behaves similarly, in some sense, to  $H^p_{2\pi}$. 
Consequently,  there is no larger B.f.s.\   containing
$L^p(-1,1)$ such that $T_p$ has an $L^p(-1,1)$-valued extension to this space,
\cite[Example 4.21]{okada-ricker-sanchez}. 
However, for $p=2$ the situation is significantly different,
as already pointed out long ago in \cite[p.44]{sohngen}.
One of the reasons is that the range of $T_2$ is a proper dense subspace of $L^2(-1,1)$.
The  arguments used for $T_p$ in the cases $1<p<2$ and $2<p<\infty$ 
do not apply to $T_2$. Moreover, they fail to indicate whether or not 
$T_2$ has an $L^2(-1,1)$-valued extension to a B.f.s.\ genuinely larger than $L^2(-1,1)$.
The atypical behavior of $T$ when $p=2$ has also been observed
in \cite{astala-etal}, where $T$ is considered to be acting in weighted $L^p$-spaces.
Accordingly, the case $p=2$ requires  different arguments.

In this paper we consider the inversion and the extension of the finite 
Hilbert transform $T$  on function spaces on $(-1,1)$. 
In Section \ref{S3} we extend known properties
of $T$ when it acts on the spaces $L^p(-1,1)$, for $p\not=2$, to a larger class of r.i.\ spaces $X$ on  
$(-1,1)$ satisfying certain restrictions on their Boyd indices, more precisely,
that $0<\underline{\alpha}_X\le\overline{\alpha}_X<1/2$ or
$1/2<\underline{\alpha}_X\le\overline{\alpha}_X<1$; 
see Theorems \ref{theo-3} and \ref{theo-4}.
In particular, it is established that $T$ is a Fredholm operator
in such r.i.\ spaces.
This allows a refinement of the solution of the airfoil equation by extending it to 
such  r.i.\ spaces; see Corollary \ref{cor-airfoil}.
In Section \ref{S4} we apply the results of the previous section to 
prove (cf. Theorem \ref{theo-10})
the impossibility of extending the finite 
Hilbert transform when it acts on r.i.\ spaces $X$ satisfying 
$0<\underline{\alpha}_X\le\overline{\alpha}_X<1/2$ or
$1/2<\underline{\alpha}_X\le\overline{\alpha}_X<1$. 
The proof relies on a deep result of Talagrand concerning $L^0$-valued measures.
In the course of that investigation we establish a rather unexpected 
characterization of when a function $f\in L^1(-1,1)$  belongs to
$X$ in terms of the set of $T$-transforms $\{T(f\chi_A): A\; \mathrm{measurable}\}$; see
Proposition \ref{cor-8}. In the final Section \ref{S5} we address the case $p=2$.
It is established (cf. Theorem \ref{theo-14}),
via a completely different approach, that $T\colon L^2(-1,1)\to L^2(-1,1)$
does not have a continuous $L^2(-1,1)$-valued extension to any larger B.f.s. 
  The argument relies on showing that the norm
\begin{equation*}
f\mapsto \sup_{|\theta|=1}\left\|T(\theta f)\right\|_2
\end{equation*}
(equivalent to \eqref{fourier} in the appropriate setting) 
is equivalent to the usual norm in $L^2(-1,1)$. We conclude
Section \ref{S5} by extending the above mentioned characterization to show that
$f\in L^2(-1,1)$ if and only if $T(f\chi_A)\in L^2(-1,1)$ for every measurable set 
$A\subseteq(-1,1)$; see Corollary \ref{cor-15}.

Not  all r.i.\ spaces $X$ which $T$ maps 
into itself (i.e., satisfying $0<\underline{\alpha}_X\le\overline{\alpha}_X<1$) 
are covered. Except when $X=L^2(-1,1)$, for those r.i.\ spaces $X$
not satisfying the conditions $0<\underline{\alpha}_X\le\overline{\alpha}_X<1/2$ or
$1/2<\underline{\alpha}_X\le\overline{\alpha}_X<1$ (e.g., the 
Lorentz spaces $L^{2,q}$ for $1\le q\le\infty$ with $q\not=2$) 
the techniques used here do not apply; see Remark \ref{rem-final}.


\section{Preliminaries}
\label{S2}


In this paper the relevant measure space  is $(-1,1)$ 
equipped with its Borel $\sigma$-algebra $\mathcal{B}$ and  Lebesgue measure $|\cdot|$
(restricted to $\mathcal{B}$).
We  denote by $\text{sim }\mathcal{B}$ the vector space of 
all $\mathbb{C}$-valued, $\mathcal{B}$-simple functions and by
$L^0(-1,1)=L^0$ the space (of equivalence classes) of all $\mathbb{C}$-valued
measurable functions, endowed with the topology of convergence in measure.
The space $L^p(-1,1)$ is denoted simply by $L^p$, for $1\le p\le\infty$.

A \textit{Banach function space} (B.f.s.) $X$ on  $(-1,1)$ is a
Banach space  $X\subseteq L^0$ satisfying
the ideal property, that is, $g\in X$ and $\|g\|_X\le\|f\|_X$
whenever $f\in X$ and $|g|\le|f|$ a.e.  
The \textit{associate space} $X'$  of $X$ consists  of all
functions $g$ satisfying $\int_{-1}^1|fg|<\infty$, for every
$f\in X$, equipped with the norm
$\|g\|_{X'}:=\sup\{|\int_{-1}^1fg|:\|f\|_X\le1\}$. 
The space $X'$ is a closed subspace of the Banach space dual $X^*$ of $X$. 
The second  associate space $X''$ 
of $X$ is defined as $X''=(X')'$. The norm in $X$ is absolutely continuous if,
for every $f\in X$, we have $\|f\chi_A\|_X\to0$ whenever $|A|\to0$.
The space $X$ satisfies the Fatou property  if, whenever  $\{f_n\}_{n=1}^\infty\subseteq X$ satisfies
$0\le f_n\le f_{n+1}\uparrow f$ a.e.\ with $\sup_n\|f_n\|_X<\infty$,
then $f\in X$ and $\|f_n\|_X\to\|f\|_X$.

A \textit{rearrangement invariant} (r.i.) space $X$ on $(-1,1)$ is a
B.f.s.\  
such that if $g^*\le f^*$ with $f\in X$,  
then $g\in X$ and $\|g\|_X\le\|f\|_X$.
Here $f^*\colon[0,2]\to[0,\infty]$ is 
the decreasing rearrangement of $f$, that is, the
right continuous inverse of its distribution function:
$\lambda\mapsto|\{t\in (-1,1):\,|f(t)|>\lambda\}|$.
The associate space $X'$ of a r.i.\ space $X$ is again a r.i.\ space.
Every r.i.\ space on $(-1,1)$ satisfies 
$L^\infty\subseteq X\subseteq L^1$, 
\cite[Corollary II.6.7]{bennett-sharpley}. 
Moreover, if $f\in X$ and $g\in X'$, then $fg\in L^1$ and
$\|fg\|_{L^1}\le \|f\|_X \|g\|_{X'}$, i.e., H\"older's inequality
is available. 
The fundamental function of $X$ is defined by 
$\varphi_X(t):=\|\chi_{A}\|_X$ for $A\in\mathcal{B}$ 
with $|A|=t$, for $t\in[0,2]$.

In this paper \textit{all} B.f.s.' $X$ (hence, all r.i.\ spaces) are on
$(-1,1)$ relative to Lebesgue measure and, as in 
\cite{bennett-sharpley},    
satisfy the Fatou property. In this case $X''=X$ and hence,
$f\in X$ if and only if $\int_{-1}^1|fg|<\infty$, for every $g\in X'$.
Moreover, $X'$ is a norm-fundamental subspace of $X^*$, that is, 
$\|f\|_X=\sup_{\|g\|_{X'}\le1} |\int_{-1}^1fg|$ for $f\in X$, \cite[pp.12-13]{bennett-sharpley}.
If $X$ is separable, then $X'=X^*$.

The family of r.i.\ spaces includes many classical spaces 
appearing in analysis, such as  the Lorentz $L^{p,q}$ spaces, 
\cite[Definition IV.4.1]{bennett-sharpley}, Orlicz $L^\varphi$ spaces 
\cite[\S4.8]{bennett-sharpley}, Marcinkiewicz $M_\varphi$ spaces, 
\cite[Definition II.5.7]{bennett-sharpley}, Lorentz $\Lambda_\varphi$ spaces,
\cite[Definition II.5.12]{bennett-sharpley},
and the Zygmund $L^p(\text{log L})^\alpha$ spaces, 
\cite[Definition IV.6.11]{bennett-sharpley}. In particular, 
   $L^p=L^{p,p}$, for $1\le p\le \infty$.  
The space weak-$L^1$, denoted by $L^{1,\infty}(-1,1)=L^{1,\infty}$, will play
an important role; it is not a Banach space,
\cite[Definition IV.4.1]{bennett-sharpley}. It satisfies
$L^1\subseteq L^{1,\infty} \subseteq L^0$, with all
inclusions continuous.

The dilation operator $E_t$ for $t>0$ is defined, for 
each $f\in X$, by $E_t(f)(s):=f(st)$ for $-1\le st\le1$ and zero in 
other cases. The operator $E_t\colon X\to X$  is bounded 
with $\|E_t\|_{X\to X}\le \max\{t,1\}$. The \textit{lower} and \textit{upper 
Boyd indices} of  $X$ are defined, respectively, by
\begin{equation*}
\underline{\alpha}_X\,:=\,\sup_{0<t<1}\frac{\log \|E_{1/t}\|_{X\to X}}{\log t}
\;\;\mbox{and}\;\;
\overline{\alpha}_X\,:=\,\inf_{1<t<\infty}\frac{\log \|E_{1/t}\|_{X\to X}}{\log t} ,
\end{equation*}
\cite[Definition III.5.12]{bennett-sharpley}. 
They satisfy $0\le\underline{\alpha}_X\le \overline{\alpha}_X\le1$.
Note that $\underline{\alpha}_{L^p}= \overline{\alpha}_{L^p}=1/p$.

We recall a  technical fact from 
the theory of r.i.\ spaces that will be often used; 
see, for example, \cite[Proposition 2.b.3]{lindenstrauss-tzafriri}.

\begin{lemma}\label{lemma-1}
Let $X$ be a r.i.\ space  such that 
$0<\alpha<\underline{\alpha}_X\le \overline{\alpha}_X<\beta<1$.
Then there exist $p,q$ satisfying $1/\beta<p<q<1/\alpha$ such that
$L^q\subseteq X \subseteq L^p$ with
continuous  inclusions.
\end{lemma}

An important role will be played by the Marcinkiwiecz space 
$L^{2,\infty}(-1,1)=L^{2,\infty}$, also known as  weak-$L^2$, \cite[Definition IV.4.1]{bennett-sharpley}. 
It consists of those $f\in L^0$  satisfying
\begin{equation}\label{L2oo}
f^*(t)\le \frac{M}{t^{1/2}},\quad 0<t\le2,
\end{equation}
for some constant $M>0$. Consider the function $1/\sqrt{1-x^2}$ on $(-1,1)$. 
Since its decreasing rearrangement  $(1/\sqrt{1-x^2})^*$
is the function $t\mapsto 2/t^{1/2}$,  it follows that $1/\sqrt{1-x^2}$ belongs to
$L^{2,\infty}$. Actually, for any r.i.\ space $X$  it is the case that 
$1/\sqrt{1-x^2}\in X$ if and only if $L^{2,\infty}\subseteq X$. 
Consequently, $L^{2,\infty}$ is the \textit{smallest} r.i.\ space  
which contains  $1/\sqrt{1-x^2}$. Note that  
$\underline{\alpha}_{L^{2,\infty}}= \overline{\alpha}_{L^{2,\infty}}=1/2$.

For all of the above and further facts on r.i.\  spaces see \cite{bennett-sharpley}, 
\cite{lindenstrauss-tzafriri}, for example.


\section{Inversion of the finite Hilbert transform on r.i.\ spaces}
\label{S3}


In  \cite[Ch.11]{king}, \cite{okada-elliot}, \cite[\S4.3]{tricomi} 
a detailed study of the inversion of the finite Hilbert 
transform was undertaken for $T$ acting on the spaces $L^p$ whenever 
$1<p<2$ and $2<p<\infty$. 
We study here the extension of those results to  a larger class of spaces, 
namely,  the r.i.\ spaces. The restrictions on $p$ indicated above for 
the $L^p$ spaces can  be formulated  for r.i.\ spaces in terms 
of their Boyd indices, namely, $0<\underline{\alpha}_X\le\overline{\alpha}_X<1/2$
and $1/2<\underline{\alpha}_X\le\overline{\alpha}_X<1$.

A result of Boyd, \cite[Theorem III.5.18]{bennett-sharpley}, 
allows the extension of  Riesz's classical theorem 
on the boundedness of the Hilbert 
transform  $H$ on the  spaces $L^p(\mathbb{R})$, for $1<p<\infty$, to 
a certain class of r.i.\ spaces. Indeed,
since $Tf=\chi_{(-1,1)}H(f\chi_{(-1,1)})$, it follows for 
a r.i.\ space $X$  with non-trivial lower and upper Boyd indices, that is, 
$0<\underline{\alpha}_X\le \overline{\alpha}_X<1$, that
$T\colon X\to X$ boundedly; this is indicated by simply writing $T_X$. 
Since $\underline{\alpha}_{X'}=1-\overline{\alpha}_X$
and $\overline{\alpha}_{X'}=1-\underline{\alpha}_X$, the condition 
$0<\underline{\alpha}_X\le \overline{\alpha}_X<1$ implies that 
$0<\underline{\alpha}_{X'} \le \overline{\alpha}_{X'}<1$. Hence,
 $T_{X'}\colon X'\to X'$
is also bounded. The operator $T$ is not continuous on $L^1$. However, 
due to a result of Kolmogorov, 
\cite[Theorem III.4.9(b)]{bennett-sharpley}, $T\colon L^1\to L^{1,\infty}$ is continuous. 
It follows  from
the Parseval formula in Proposition \ref{prop-2}(b) 
below that the restriction of the dual operator 
$T_X^*\colon X^*\to X^*$ of $T_X$ to the closed subspace $X'$ of $X^*$ 
is precisely $-T_{X'}\colon X'\to X'$.


In the study of the operator $T$ an important role is played by the particular
function $1/\sqrt{1-x^2}$, which belongs to each $L^p$, $1\le p<2$. 
The reason is that 
\begin{equation}\label{0}
T\Big(\frac{1}{\sqrt{1-x^2}}\Big)(t)=
p.v. \frac{1}{\pi}\int_{-1}^{1}\frac{1}{\sqrt{1-x^2} (x-t)}\,dx=0,\quad -1< t< 1,
\end{equation}
and, moreover, that if $T(f)(t)=0$ for a.e.\ $t\in(-1,1)$ with 
$f$ a function belonging to some space 
$L^p$, $1<p<\infty$, then necessarily $f(x)=C/\sqrt{1-x^2}$ for some 
constant $C\in\C$; \cite[\S4.3 (14)]{tricomi}. 
Combining this observation with Lemma \ref{lemma-1} it follows,
for every r.i.\ space $X$ satisfying 
$0<\underline{\alpha}_X\le\overline{\alpha}_X<1$,
that $T_X$ is either injective or  $\mathrm{dim}(\mathrm{Ker}(T_X))=1$.
Recall that 
$L^{2,\infty}$ is the \textit{smallest} r.i.\ space  
containing the function $1/\sqrt{1-x^2}$,  that is, 
$1/\sqrt{1-x^2}\in X$ if and only if $L^{2,\infty}\subseteq X$.


The Parseval and Poincar\'e-Bertrand 
formulae are  important tools for studying the finite 
Hilbert transform in the spaces $L^p$, $1<p<\infty$, \cite[\S 4.3]{tricomi}. 
It should be noted that a result of Love 
is essential in order to have a sharp version of the 
Poincar\'e-Bertrand formula, \cite{love}.
The validity of both of these formulae can be extended to the setting of r.i.\ spaces.

\begin{proposition}\label{prop-2}
Let $X$ be a  r.i.\ space    satisfying
$0<\underline{\alpha}_X\le \overline{\alpha}_X<1$. 
\begin{itemize}
\item[(a)] Let $f\in L^1$ satisfy $fT_{X'}(g)\in L^1$ for all $g\in X'$. Then, for every
$g\in X'$, the function $gT(f)\in L^1$ and
$$
\int_{-1}^1fT_{X'}(g)=-\int_{-1}^1gT(f).
$$
\item[(b)] 
The Parseval formula holds for the pair $X$ and $X'$, that is,
\begin{equation*}\label{parseval}
\int_{-1}^{1}fT_{X'}(g)=-\int_{-1}^{1}gT_X(f),\quad f\in X, g\in X'.
\end{equation*}
\item[(c)] The Poincar\'e-Bertrand formula holds for the pair $X$ and $X'$, that is, 
for all $f\in X$ and $g\in X'$ we have
\begin{equation*}\label{poincare}
T(gT_X(f)+fT_{X'}(g))=(T_X(f))(T_{X'}(g))-fg,\quad \mathrm{a.e.}
\end{equation*}
\end{itemize}
\end{proposition}

\begin{proof}
(a) Assume first that $f\in L^\infty$. By Lemma \ref{lemma-1}, there exists
$1<q<\infty$ satisfying $L^q\subseteq X$, so that $X'\subseteq L^{q'}$. Then
$$
\int_{-1}^1fT_{X'}(g)=-\int_{-1}^1gT_X(f)=-\int_{-1}^1gT(f),\quad g\in X',
$$
via Parseval formula for the pair $L^q$ and $L^{q'}$, \cite[Sect. 11.10.8]{king}, 
\cite[Sect. 4.2, 4.3]{tricomi}, because 
$f\in L^\infty\subseteq L^q$ and $g\in X'\subseteq L^{q'}$.

Now let $f\in L^1$ be a general function satisfying the assumption of (a). Define  
$A_n:=|f|^{-1}([0,n])$ and $f_n:=f\chi_{A_n}\in L^\infty$ for $n\in\N$. 
Then $\lim_n f_n=f$ in $L^1$. It follows from Kolmogorov's Theorem that 
$\lim_nT(f_n)=T(f) $ in $L^{1,\infty}$. Since the inclusion $L^{1,\infty}\subseteq L^0$
is continuous, we can conclude that $\lim_nT(f_n)=T(f)$ in measure. Accordingly,
by passing to a subsequence if necessary, we 
may assume that $\lim_nT_X(f_n)=\lim_nT(f_n)=T(f)$ pointwise a.e.

Fix $g\in X'$. Given any $A\in\mathcal{B}$, the Dominated Convergence Theorem
ensures that
\begin{equation}\label{i}
\lim_n f_nT_{X'}(g\chi_A)=fT_{X'}(g\chi_A),\quad \text{in } L^1,
\end{equation}
as $|f_nT_{X'}(g\chi_A)|\le|fT_{X'}(g\chi_A)|$ pointwise for 
$n\in\N$ and because $fT_{X'}(g\chi_A)\in L^1$ by assumption.  
For each $n\in\N$, the first part of this 
proof applied to $f_n\in L^\infty\subseteq X$ yields
$\int_{-1}^1f_nT_{X'}(g\chi_A)=-\int_{-1}^1(g\chi_A)T_X(f_n)$. It follows from \eqref{i}
that
\begin{align*}\label{ii}
\lim_n\int_AgT_X(f_n)&=\lim_n\int_{-1}^1 (g\chi_A)T_X(f_n)
\nonumber
\\ & =
-\lim_n\int_{-1}^1 f_nT_{X'}(g\chi_A)
=
-\int_{-1}^1 fT_{X'}(g\chi_A).
\end{align*}
Since this holds for all sets $A\in\mathcal{B}$ and since 
$\lim_n gT_X(f_n)=gT(f)$
pointwise a.e., we can conclude that both $gT(f)\in L^1$ and 
\begin{equation}\label{iii}
\lim_n gT_X(f_n)=gT(f),\quad \text{in } L^1;
\end{equation}
see, for example, \cite[Lemma 2.3]{lewis}. 
This and \eqref{i} with $A:=(-1,1)$ ensure that 
$\int_{-1}^1fT_{X'}(g)=-\int_{-1}^1gT(f)$. So, (a) is established.

(b) Given any $f\in X$ and $g\in X'$, H\"older's inequality 
ensures that $fT_{X'}(g)\in L^1$. So, part (b) follows from (a).

(c) Fix $f\in X$ and $g\in X'$. The proof of part (a) shows that there exists
a sequence $\{f_n\}_{n=1}^\infty\subseteq L^\infty\subseteq X$
satisfying the conditions:
\begin{itemize}
\item[(i)] $\lim_n f_n=f$ and $\lim_n T_X(f_n)=T_X(f)$ pointwise a.e.,
as well as
\item[(ii)] $\lim_n f_nT_{X'}(g)=fT_{X'}(g)$ in $L^1$ 
and $\lim_n gT_X(f_n)=gT_X(f)$ in $L^1$;
\end{itemize}
see \eqref{i} with $A:=(-1,1)$ and \eqref{iii}, respectively. Condition (ii)
implies that
\begin{equation}\label{iv}
\lim_n T(gT_X(f_n)+f_nT_{X'}(g))=T(gT_X(f)+fT_{X'}(g))
\end{equation}
in $L^{1,\infty}$ (via Kolmogorov's Theorem) and 
hence, in $L^0$. On the other hand, condition (i) implies that
\begin{equation}\label{v}
\lim_n \big((T_X(f_n))(T_{X'}(g))-f_ng\big)=
(T_X(f))(T_{X'}(g))-fg
\end{equation}
pointwise a.e. As in the proof of part (a), select $1<q<\infty$
such that $L^q\subseteq X$. Since $f_n\in L^\infty\subseteq L^q$ 
for $n\in\N$ and $g\in X'\subseteq L^{q'}$,
the Poincar\'e-Bertrand formula for the pair $L^q$ 
and $L^{q'}$ gives, for each $n\in\N$, that
\begin{equation}\label{vi}
T(gT_X(f_n)+f_nT_{X'}(g)=(T_X(f_n)(T_{X'}(g))-f_ng,\quad a.e.,
\end{equation}
with the identities holding outside a null set which is independent
of $n\in\N$. In view of \eqref{iv} and \eqref{v}, take the limit
of both sides of \eqref{vi} in $L^0$ to obtain 
the identity $T(gT_X(f)+fT_{X'}(g))=(T_X(f))(T_{X'}(g))-fg$ in 
$L^0$. This is precisely the Poincar\'e-Bertrand formula for $f\in X$ and $g\in X'$.
\end{proof}


We  can now extend certain  results obtained in  \cite{okada-elliot}, \cite[\S4.3]{tricomi}  
for the spaces  $L^p$ with $1<p<2$ to the larger 
family of r.i.\ spaces satisfying  
$1/2<\underline{\alpha}_X\le \overline{\alpha}_X<1$.

For each $f\in X$ define pointwise the measurable function 
\begin{equation}\label{T-hat}
(\widehat{T}_X(f))(x):=\frac{-1}{\sqrt{1-x^2}}\, 
T_X(\sqrt{1-t^2}f(t))(x),\quad \mathrm{a.e. }\; x\in (-1,1).
\end{equation}

\begin{theorem}\label{theo-3}
Let $X$ be a  r.i.\ space  satisfying
$1/2<\underline{\alpha}_X\le \overline{\alpha}_X<1.$
\begin{itemize}
\item[(a)] $\mathrm{Ker}(T_X)$ is the 1-dimensional 
subspace of $X$ spanned by the function $1/\sqrt{1-x^2}$.
\item[(b)] The linear operator $\widehat{T}_X$ defined by \eqref{T-hat}
maps $X$ boundedly into $X$ and satisfies 
 $T_X\widehat T_X=I_X$ (the identity operator on $X$). Moreover,
\begin{equation}\label{3-b}
\int_{-1}^1(\widehat{T}_X(f))(x)\,dx=0,\quad f\in X. 
\end{equation}
\item[(c)] The operator $T_X\colon X\to X$ is surjective.
\item[(d)] The identity $\widehat T_XT_X=I_X -P_X$ holds, 
with $P_X$ the bounded projection given by
\begin{equation}\label{3-d}
f\mapsto P_X(f):=\left(\frac1\pi\int_{-1}^1 f(t)\,dt\right)
\frac{1}{\sqrt{1-x^2}},\quad f\in X.
\end{equation}
\item[(e)] The operator $\widehat{T}_X$ is  an isomorphism 
onto its range $\text{R}(\widehat{T}_X)$. Moreover,
\begin{equation}\label{3-e}
\text{R}(\widehat{T}_X)=\left\{f\in X: \int_{-1}^{1} f(x)dx=0\right\}.
\end{equation}
\item[(f)] The following decomposition of $X$ holds 
(with $\langle\cdot\rangle$ denoting linear span):
\begin{equation}\label{3-f}
X=\left\{f\in X: \int_{-1}^{1} f(x)dx=0\right\}\oplus \left\langle 
\frac{1}{\sqrt{1-x^2}}\right\rangle
=\text{R}(\widehat{T}_X) \oplus \left\langle 
\frac{1}{\sqrt{1-x^2}}\right\rangle.
\end{equation}
\end{itemize}
\end{theorem}

\begin{proof}
(a) Since $1/2<\underline{\alpha}_X$ we have $L^{2,\infty}\subseteq X$ and so 
$1/\sqrt{1-x^2}\in X$. Accordingly, $\langle 
\frac{1}{\sqrt{1-x^2}}\rangle\subseteq \text{Ker}(T_X)$. 
Conversely, let $f\in \text{Ker}(T_X)$.
By Lemma \ref{lemma-1} there is $1<p<2$ such that $f\in L^p$. 
As noted   prior to Proposition \ref{prop-2}, this implies that 
$f(x)=c/\sqrt{1-x^2}$ for some $c\in \C$.

(b) Via Lemma \ref{lemma-1} there exist $1<p<q<2$ such that 
$1/q<\underline{\alpha}_X\le \overline{\alpha}_X<1/p$ and 
$L^q\subseteq X \subseteq L^p$.  Consider the weight function
$\rho(x):= 1/\sqrt{1-x^2}$ on $(-1,1)$. Appealing to results on 
boundedness of the Hilbert transform on weighted $L^p$ spaces, $T$ is bounded
from the weighted space $L^p((-1,1),\rho)$ into itself 
and from the weighted space $L^q((-1,1),\rho)$ 
into itself,  \cite[Ch.1, Theorem 4.1]{gohberg-krupnik}.
This is equivalent to the fact that
$$
f\mapsto \widehat{T}(f):=\frac{-1}{\sqrt{1-x^2}} T_X\big(\sqrt{1-x^2}f(x)\big),
$$
is well defined on $L^p$ and bounded 
as an operator from $L^p$ into $L^p$ and from $L^q$ into $L^q$.
The condition on the indices $1/q<\underline{\alpha}_X\le \overline{\alpha}_X<1/p$
allows us to apply Boyd's interpolation theorem, 
\cite[Theorem 2.b.11]{lindenstrauss-tzafriri},
to conclude that $\widehat{T}$ maps $X$ boundedly into $X$.
According to \eqref{T-hat}, note that $\widehat{T}_X$ is the operator 
$\widehat{T}\colon X\to X$.

To establish $T_X\widehat{T}_X=I_X$, choose $1<p<2$  
such that $X\subseteq L^p$.
It follows from (2.7) on p.46 of \cite{okada-elliot} 
that $T_{L^p}\widehat{T}_{L^p}=I_{L^p}$.
Let $f\in X\subseteq L^p$. Since all three operators $T_X$, $\widehat{T}_X$ and $I_X$ map
$X$ into $X$ it follows that $T_X(\widehat{T}_X(f))=f=I_X(f)$.

To establish \eqref{3-b} let $f\in X\subseteq L^p$, with $1<p<2$ as above. 
Then \eqref{3-b} above follows from the validity of 
\eqref{3-b} in $L^p$; see (2.6) on p.46 of \cite{okada-elliot}.

(c) Follows immediately from $T_X\widehat{T}_X=I_X$.

(d) Since $(1/\pi)\int_{-1}^1dx/\sqrt{1-x^2}=1$, it follows that $P_X$ as given in 
\eqref{3-d} is indeed a linear projection from $X$ onto the 1-dimensional subspace 
$\langle \frac{1}{\sqrt{1-x^2}}\rangle\subseteq X$. 
The boundedness of $P_X$ is a consequence of H\"older's inequality
(applied to   $f=\mathbf{1}\cdot f$ with $\mathbf{1}\in X'$ and $f\in X$ fixed), namely  
$$
\|P_X(f)\|_X\le \frac1\pi \left\|\frac{1}{\sqrt{1-x^2}}\right\|_X \|\mathbf{1}\|_{X'}\|f\|_X.
$$

To verify that $P_X=I_X-\widehat{T}_XT_X$, fix $f\in X$. Then $T_X\widehat{T}_X=I_X$ 
implies that $T_X(I_X-\widehat{T}_XT_X)(f)=0$,   that is,  
$$
(I_X-\widehat{T}_XT_X)(f)\in \text{Ker}(T_X).
$$
According to part (a) there exists $c\in\C$ such that 
\begin{equation}\label{b}
(I_X-\widehat{T}_XT_X)(f)=\frac{c}{\sqrt{1-x^2}}.
\end{equation}
 But,   $\int_{-1}^1\widehat{T}_X(T_X(f))(x)\,dx=0$ (by \eqref{3-b}) and so \eqref{b}
 implies that
$$
\int_{-1}^1f(x)\,dx=c\int_{-1}^1dx/\sqrt{1-x^2}=c\pi,
$$ 
that is, $c=(1/\pi) \int_{-1}^1f(x)\,dx$. So, again 
by \eqref{b}, we can conclude that $(I_X-\widehat{T}_XT_X)(f)=P_X(f)$.  
Since $f\in X$ is arbitrary, it follows that $I_X-\widehat{T}_XT_X=P_X$.

(e) The identity $T_X\widehat{T}_X=I_X$ implies that $\widehat{T}_X$ is injective. So, 
$\widehat{T}_X\colon X\to R(\widehat{T})$ is a linear bijection.

To verify \eqref{3-e} suppose $f\in X$ 
satisfies $\int_{-1}^1f(x)\,dx=0$, i.e.,   $P_X(f)=0$. Then the identity
$\widehat{T}_XT_X=I_X-P_X$ shows that $f=\widehat{T}_X(h)$ 
with $h:=T_X(f)\in X$,  i.e., $f\in R(\widehat{T}_X)$. 
Conversely, suppose that $f=\widehat{T}_X(g)\in R(\widehat{T}_X)$ for 
some $g\in X$. Then $g=T_X(f)$   as $T_X\widehat{T}_X=I_X$. Accordingly, 
$$
f=\widehat{T}_X(g)=\widehat{T}_XT_X(f)=I_X(f)-P_X(f)=f-P_X(f)
$$
  and so $P_X(f)=0$. It is then clear from \eqref{3-d} 
that $\int_{-1}^1f(x)\,dx=0$, i.e., $f$ belongs to the right-side
of \eqref{3-e}. This establishes \eqref{3-e}.

  Since the linear functional
$f\mapsto \varphi_1(f):=\int_{-1}^1f(x)\,dx$, for $f\in X$, 
belongs to $X^*$, as $\mathbf{1}\in X'\subseteq X^*$, it follows via \eqref{3-e} that 
$R(\widehat{T}_X)=\text{Ker}(\varphi_1)$ and hence,
$R(\widehat{T}_X)$ is a \textit{closed} subspace of $X$. Accordingly, 
$\widehat{T}_X\colon X\to R(\widehat{T}_X)$ is a Banach space isomorphism.

(f) The identity $\widehat{T}_XT_X+P_X=I_X$ shows that 
each $f\in X$ has the form $f=\widehat{T}_X(T_X(f))+P_X(f)$ with 
$\widehat{T}_X(T_X(f))\in R(\widehat{T}_X)$ and, via \eqref{3-d}, 
$P_X(f)\in \langle1/\sqrt{1-x^2}\rangle$. 
So, it remains to show that the decomposition in \eqref{3-f} 
is a \textit{direct sum}. To this effect, let 
$h\in R(\widehat{T}_X)\cap \langle1/\sqrt{1-x^2}\rangle$, in which case $h=\widehat{T}_X(f)$ 
for some $f\in X$ and $h=c/\sqrt{1-x^2}$ for some $c\in\C$, 
that is, $\widehat{T}_X(f)=c/\sqrt{1-x^2}$. Integrating
both sides of this identity over $(-1,1)$ and appealing to 
\eqref{3-b} shows that $c=0$. Hence, $h=0$.
\end{proof}


Next we extend   certain results obtained in    \cite{okada-elliot}, \cite[\S4.3]{tricomi},  for the  
spaces $L^p$ with $2<p<\infty$, 
to the larger family of r.i.\ spaces $X$ satisfying 
$0<\underline{\alpha}_X\le \overline{\alpha}_X<1/2$.
Then $1/2<\underline{\alpha}_{X'}\le \overline{\alpha}_{X'}<1$ 
and so $1/\sqrt{1-x^2}\in X'$. Hence, for every
$f\in X$, the function $f(x)/\sqrt{1-x^2}\in L^1$. Accordingly, 
we can define  pointwise the measurable function  
\begin{equation}\label{T-check}
(\widecheck{T}_X(f))(x):=-\sqrt{1-x^2}\, 
T\Big(\frac{f(t)}{\sqrt{1-t^2}}\Big)(x),\quad \mathrm{a.e. } \;x\in(-1,1).
\end{equation}

\begin{theorem}\label{theo-4}
Let $X$ be a  r.i.\ space  satisfying
$0<\underline{\alpha}_X\le \overline{\alpha}_X<1/2.$
\begin{itemize}

\item[(a)] The operator $T_X\colon X\to X$ is injective.

\item[(b)] The linear operator $\widecheck{T}_X$ defined by \eqref{T-check}
is bounded from  $X$ into $X$ and satisfies $\widecheck T_XT_X=I_X$.

\item[(c)] The identity $T_X\widecheck T_X=I_X -Q_X$ holds, with $Q_X$ 
the bounded projection  given by
\begin{equation}\label{B}
f\in X\mapsto Q_X(f):=\left(\frac1\pi\int_{-1}^1 
\frac{f(x)}{\sqrt{1-x^2}}\,dx\right) \mathbf{1}.
\end{equation}

\item[(d)] The range of $T_X$ is the closed subspace of $X$ given by
\begin{equation}\label{C}
R(T_X)=\left\{f\in X: \int_{-1}^{1} \frac{f(x)}{\sqrt{1-x^2}}dx=0\right\}
=\mathrm{Ker}(Q_X).
\end{equation}
Moreover,  $\widecheck T_X$ is an isomorphism from $\text{R}(T_X)$ onto $X$.

\item[(e)] The following decomposition of $X$ holds:
\begin{equation}\label{D}
X=\left\{f\in X: \int_{-1}^{1} \frac{f(x)}{\sqrt{1-x^2}}dx=0\right\}\oplus \left\langle 
\mathbf{1}\right\rangle= R(T_X)\oplus \left\langle 
\mathbf{1}\right\rangle.
\end{equation}
\end{itemize}
\end{theorem}

\begin{proof}
(a) Since $\overline{\alpha}_X<1/2$ we have that $X\subsetneqq L^{2,\infty}$ 
and so $1/\sqrt{1-x^2}\notin X$. Hence, $T_X$ is injective; see the discussion after 
\eqref{0}.

(b) Via Lemma \ref{lemma-1} there exist $2<p<q<\infty$ such that 
$1/q<\underline{\alpha}_X\le \overline{\alpha}_X<1/p$ and 
$L^q\subseteq X \subseteq L^p$.  Consider the weight function
$\rho(x):= \sqrt{1-x^2}$ on $(-1,1)$. Appealing again to results on 
boundedness of the Hilbert transform on weighted $L^p$ spaces, $T$ is bounded
from the weighted space $L^p((-1,1),\rho)$ into itself 
and from the weighted space $L^q((-1,1),\rho)$ 
into itself, \cite[Ch.1 Theorem 4.1]{gohberg-krupnik}.
This is equivalent to the fact that
$$
f\mapsto \widecheck{T}(f):= -\sqrt{1-x^2}\,T\Big(\frac{f(x)}{\sqrt{1-x^2}}\Big),
$$
is well defined on $L^p$ and bounded  
as an operator from $L^p$ into $L^p$ and from $L^q$ into $L^q$.
The condition on the indices $1/q<\underline{\alpha}_X\le \overline{\alpha}_X<1/p$
allows us to apply Boyd's interpolation theorem, 
\cite[Theorem 2.b.11]{lindenstrauss-tzafriri},
to deduce that $\widecheck{T}$ maps $X$ boundedly  into $X$. 
According to \eqref{T-check} note that  $\widecheck{T}_X$ is the 
operator $\widecheck{T}\colon X\to X$.

To establish $\widecheck T_XT_X=I_X$, recall that $X\subseteq L^p$.
It follows from (2.10) on p.48 of  \cite{okada-elliot} that $\widecheck T_{L^p}T_{L^p}=I_{L^p}$.
Let $f\in X\subseteq L^p$. Since all three operators $T_X$, $\widecheck{T}_X$ and $I_X$ map
$X$ into $X$ it follows that $\widecheck T_X(T_X(f))=f=I_X(f)$.

(c) It is routine to check that $Q_X$ is a linear projection onto 
the 1-dimensional space $\langle\mathbf{1}\rangle$.
Since $g(x)=1/\sqrt{1-x^2}\in X'$, the boundedness of $Q_X$ follows 
from \eqref{B} via H\"older's inequality, namely
$$
\|Q_X(f)\|_X\le\frac1\pi\|g\|_{X'} \|\mathbf{1}\|_X\|f\|_X,\quad f\in X.
$$

To establish the identity $T_X\widecheck{T}_X=I_X-Q_X$, 
choose $2<p<\infty$  such that $X\subseteq L^p$.
It follows from (2.11) on p.48 of  \cite{okada-elliot} that 
$T_{L^p}\widecheck{T}_{L^p}=I_{L^p}-Q_{L^p}$.
Let $f\in X\subseteq L^p$. Since all four operators $T_X$, $\widecheck{T}_X$, 
$Q_X$ and $I_X$ map
$X$ into $X$ it follows that $T_X(\widecheck{T}_X(f))=f-Q_X(f)=(I_X-Q_X)(f)$.

(d) Using the identities $\widecheck{T}_XT_X=I_X$ and $T_X\widecheck{T}_X=I_X-Q_X$ 
one can argue
as on p.48 of \cite{okada-elliot} to verify the identity \eqref{C}. 
In particular, since $Q_X$ is bounded, it
follows that $R(T_X)=\text{Ker}(Q_X)$ is a \textit{closed} subspace of $X$. It is clear from 
$\widecheck{T}_XT_X=I_X$ that $\widecheck{T}_X$ maps $R(T_X)$ onto $X$ and 
also that $\widecheck{T}_X$
restricted to $R(T_X)$ is injective, i.e., 
$\widecheck{T}_X\colon R(T_X)\to X$ is a linear 
bijection and bounded. By the Open Mapping Theorem 
$\widecheck{T}_X\colon R(T_X)\to X$ is actually a Banach space isomorphism.

(e) As $Q_X$ is a bounded projection, we have $X=\text{Ker}(Q_X)\oplus R(Q_X)$.
But, $\text{Ker}(Q_X)=R(T_X)$ by part   (d) and 
$R(Q_X)=\langle\mathbf{1}\rangle$ by part   (c). 
The direct sum decomposition \eqref{D} is then immediate. 
\end{proof}


\begin{remark}
Let $X$ be a  r.i.\ space satisfying 
$0<\underline{\alpha}_{X}\le \overline{\alpha}_{X}<1/2$
or $1/2<\underline{\alpha}_X\le \overline{\alpha}_X<1$. 
Then $T_X\colon X\to X$ is a Fredholm operator, that is,
$\text{dim}( \text{Ker}(T_X))<\infty$, the range $R(T_X)$ is
a closed subspace of $X$ and 
$\text{dim}( X/R(T_X))<\infty$. This holds when
$1/2<\underline{\alpha}_X\le \overline{\alpha}_X<1$ because
$\text{dim}( \text{Ker}(T_X))=1$ and $T_X$ is surjective;
see Theorem \ref{theo-3}(a), (c). The operator $T_X$ is also
Fredholm when $0<\underline{\alpha}_{X}\le \overline{\alpha}_{X}<1/2$
because it  is injective, $R(T_X)$ is closed in $X$ and 
$\text{dim}( X/R(T_X))=1$; see (a), (d), (e) of Theorem \ref{theo-4}.
\end{remark}


A  consequence of Theorems \ref{theo-3} and \ref{theo-4} is 
the possibility to extend the results in 
\cite[Ch.11]{king}, \cite{okada-elliot}, \cite[\S4.3]{tricomi},  
concerning the \textit{inversion} of the 
airfoil equation 
\begin{equation}\label{airfoil}
(T(f))(t)=p.v. \frac{1}{\pi} \int_{-1}^{1}\frac{f(x)}{x-t}\,dx=g(t),
\quad \mathrm{a.e. }\; t\in(-1,1),
\end{equation}
within the class of $L^p$-spaces for $1<p<\infty$, $p\not=2$ (with $g\in L^p$ given), 
to the significantly larger class of r.i.\  spaces $X$ whose Boyd indices satisfy
$0<\underline{\alpha}_X\le \overline{\alpha}_X<1/2$ or 
$1/2<\underline{\alpha}_X\le \overline{\alpha}_X<1.$

\begin{corollary}\label{cor-airfoil}
Let $X$ be a  r.i.\ space. 
\begin{itemize}
\item[(a)] Suppose that 
$1/2<\underline{\alpha}_X\le \overline{\alpha}_X<1$ and $g\in X$ is fixed. Then all 
solutions  $f\in X$ of the airfoil equation \eqref{airfoil} are given by  
\begin{equation}\label{E}
f(x)=\frac{-1}{\sqrt{1-x^2}}\; T_X\left(\sqrt{1-t^2}g(t)\right) (x)
+ \frac{\lambda}{\sqrt{1-x^2}},\quad \mathrm{a.e. } \;x\in(-1,1),
\end{equation}
with $\lambda\in\mathbb{C}$ arbitrary.
\item[(b)] Suppose that  
$0<\underline{\alpha}_X\le \overline{\alpha}_X<1/2$ and $g\in X$ satisfies 
$\int_{-1}^{1} \frac{g(x)}{\sqrt{1-x^2}}dx=0.$
Then there is a unique solution $f\in X$ of the airfoil equation \eqref{airfoil}, namely  
$$
f(x):=-\sqrt{1-x^2}\; T_X\left(\frac{g(t)}{\sqrt{1-t^2}}\right)(x),\quad \mathrm{a.e. } \;x\in(-1,1).
$$
\end{itemize}
\end{corollary}

\begin{proof}
(a) In this case $1/\sqrt{1-x^2}\in X$. Given any $\lambda\in\mathbb{C}$ define the function  
$$
f(x):=\frac{-1}{\sqrt{1-x^2}}\; T_X\left(\sqrt{1-t^2}g(t)\right) (x)
+ \frac{\lambda}{\sqrt{1-x^2}}
=\widehat{T}_X(g)(x)+\frac{\lambda}{\sqrt{1-x^2}}.
$$
Then the identities   $T_X\widehat{T}_X(g)=g$ and $T_X(\lambda/\sqrt{1-x^2})=0$ 
(see Theorem \ref{theo-3}) imply that $T_X(f)=g$.

Conversely, suppose that $f\in X$ satisfies   $T_X(f)=g$. 
It follows from $\widehat{T}_XT_X=I_X-P_X$ that 
$f-P_X(f)=\widehat{T}_X(g)$. By \eqref{3-d} there exists 
$\lambda\in\mathbb{C}$ such that $P_X(f)=\lambda/\sqrt{1-x^2}$
and hence, $f=\widehat{T}_X(g)+ \frac{\lambda}{\sqrt{1-x^2}}$. 
So, all solutions of the airfoil equation are indeed given by \eqref{E}.

(b) Define $f(x):=-\sqrt{1-x^2}\, T(g(t)/\sqrt{1-t^2})(x)=\widecheck{T}_X(g)$. 
By Theorem \ref{theo-4}(c) we have
$$
T_X(f)=T_X\widecheck{T}_X(g)=g-Q_X(g).
$$
But, the hypothesis on $g\in X$ implies, via \eqref{C}, 
that $g\in\text{Ker}(Q_X)$ and so $T_X(f)=g$. The uniqueness
of the solution $f$ is immediate as $T_X$ is injective (by Theorem \ref{theo-4}(a)).
\end{proof}

\begin{remark} 
The conditions 
$0<\underline{\alpha}_{X}\le \overline{\alpha}_{X}<1/2$
or $1/2<\underline{\alpha}_X\le \overline{\alpha}_X<1$ are not always satisfied, e.g.,
if $X=L^{2,q}$ with $1\le q\le\infty$. There also exist r.i.\ spaces $X$ such that 
$\underline{\alpha}_{X}<1/2< \overline{\alpha}_{X}$; see \cite[pp. 177--178]{bennett-sharpley}.
\end{remark}


\section{Extension of the finite Hilbert transform on r.i.\ spaces}
\label{S4}


  The finite Hilbert transform T$\colon L^1\to L^{1,\infty}$
has the property that $T(L^1)\not\subseteq L^1$. Hence, for any r.i.\
space $X$  we necessarily have $T(L^1)\not\subseteq X$. On the other hand,
if $X$ satisfies $0<\underline{\alpha}_X\le \overline{\alpha}_X<1$, then
$T(X)\subseteq X$ continuously. Do there exist any other B.f.s.'
$Z\subseteq L^1$ such that $X\subsetneqq Z$ and $T(Z)\subseteq X$?
As a  consequence of Theorems \ref{theo-3} and \ref{theo-4}, 
for those r.i.\ spaces $X$ satisfying
$1/2<\underline{\alpha}_X\le \overline{\alpha}_X<1$ or
$0<\underline{\alpha}_X\le \overline{\alpha}_X<1/2$,
the answer is shown to be negative; see Theorem \ref{theo-10}.


 The proof of the following result uses important facts from
the theory of vector measures, namely, a theorem of Talagrand  concerning 
$L^0$-valued measures and the Dieudonn\'e-Grothendieck Theorem for bounded 
vector measures.

\begin{proposition}\label{prop-7}
Let $X$ be a  r.i.\ space  satisfying
$0<\underline{\alpha}_X\le \overline{\alpha}_X<1$. 
Let $f\in L^1$. The following conditions are equivalent.
\begin{itemize}
\item[(a)] $T(f\chi_A)\in X$ for every $A\in\mathcal{B}$.
\item[(b)] $\displaystyle \sup_{A\in\mathcal{B}}\|T(f\chi_A)\|_X<\infty.$
\item[(c)] $T(h)\in X$ for every $h\in L^0$ with $|h|\le |f|$ a.e.
\item[(d)] $\displaystyle \sup_{|h|\le|f|}\|T(h)\|_X<\infty.$
\item[(e)] $T(\theta f)\in X$ for every $\theta\in L^\infty$ with $|\theta|=1$ a.e.
\item[(f)] $\displaystyle \sup_{|\theta|=1}\|T(\theta f)\|_X<\infty.$
\item[(g)] $fT_{X'}(g)\in L^1$ for every $g\in X'$. 
\end{itemize}
Moreover, if any one of $\mathrm{(a)}$-$\mathrm{(g)}$ is satisfied, then
\begin{equation}\label{norms}
\sup_{A\in\mathcal{B}}\big\|T(\chi_A f)\big\|_X
\le  
\sup_{|\theta|=1}\big\|T(\theta f)\big\|_X
\le 
\sup_{|h|\le|f|}\big\|T(h)\big\|_X
\le  
4 \sup_{A\in\mathcal{B}}\big\|T(\chi_A f)\big\|_X.
\end{equation}
\end{proposition}

\begin{proof}
(a)$\Rightarrow$(b). Consider the $X$-valued, finitely additive measure
\begin{equation}\label{nu}
\nu\colon A\mapsto T(f\chi_A),\quad A\in\mathcal {B}.
\end{equation}
Let $J_X\colon X\to L^0$ denote the natural continuous 
linear embedding. Then the composition 
$J_X\circ \nu\colon\mathcal{B}\to L^0$ is $\sigma$-additive. 
To establish this let $A_n\downarrow\emptyset$ in $\mathcal{B}$. 
Then $\lim_nf\chi_{A_n}=0$ in $L^1$ and hence,
$\lim_nT(f\chi_{A_n})=0$ in $L^{1,\infty}$ by Kolmogorov's Theorem.
Since $L^{1,\infty}\subseteq L^0$ continuously, we also have 
$\lim_nT(f\chi_{A_n})=0$ in $L^0$. Consequently, $\lim_n(J_X\circ\nu)({A_n})=0$ in $L^0$
which verifies the $\sigma$-additivity of $J_X\circ\nu$.

It follows from a result of Talagrand, \cite[Theorem B]{talagrand}, that there exists 
a non-negative function $\Psi_0\in L^0$ and a $\sigma$-additive vector measure
$\mu_0\colon\mathcal{B}\to L^2$ such that
$$
(J_X\circ\nu)({A})=\Psi_0\cdot \mu_0(A),\quad A\in\mathcal{B},
$$
where $\Psi_0\cdot \mu_0(A)$ is the pointwise product of two functions in $L^0$. Define 
$B_0:=\Psi_0^{-1}(\{0\})$. Then $\Psi:=\Psi_0+\chi_{B_0}\in L^0$ is strictly positive.
Consider the $L^2$-valued vector measure
$$
\mu\colon A\mapsto \chi_{(-1,1)\setminus B_0}\cdot \mu_0(A),\quad A\in\mathcal{B}.
$$
For every $A\in\mathcal{B}$, we claim that $(J_X\circ\nu)({A})=\Psi\cdot\mu(A)$. 
This follows from
\begin{align*}
\Psi\cdot\mu(A)&=(\Psi_0+\chi_{B_0})\cdot\chi_{(-1,1)\setminus B_0}\cdot \mu_0(A)
\\
&=
\chi_{(-1,1)\setminus B_0}\cdot \Psi_0\cdot \mu_0(A)
\\
&=
\chi_{(-1,1)\setminus B_0}\cdot (J_X\circ\nu)({A}) + \chi_{ B_0}\cdot (J_X\circ\nu)({A}) 
\\
&=
 (J_X\circ\nu)({A}),
\end{align*} 
where we have used $\chi_{ B_0}\cdot (J_X\circ\nu)({A}) 
=\chi_{ B_0}\cdot \Psi_0\cdot \mu({A})=0$.

Set $B_n:=\{x\in(-1,1): (n-1)<1/\Psi(x)\le n\}$, for $n\in\mathbb{N}$. Then the subset
\begin{equation}\label{c2}
\big\{\chi_{B_n\cap B}/\Psi:n\in\mathbb{N}, B\in\mathcal{B}\big\}
\end{equation}
of $L^\infty\subseteq X'\subseteq X^*$ is total for $X$. To verify this, let $g\in X$ satisfy
$$
\int_{-1}^1g(x)\chi_{B_n\cap B}(x)/\Psi(x)\,dx=0,\quad n\in\mathbb{N}, B\in\mathcal{B}.
$$
Then, for every $n\in\mathbb{N}$, the function 
$(g\chi_{B_n}/\Psi)\in X\subseteq L^1$ is 0 a.e. 
Since $1/\Psi$ is strictly positive on $(-1,1)=\cup_{n=1}^\infty B_n$, we have $g=0$ a.e. 
This implies that the subset \eqref{c2} of $X^*$ is total for $X$.

Fix $n\in\mathbb{N}$ and $B\in\mathcal{B}$. Then the scalar-valued set function
$A\mapsto\langle\nu(A),\chi_{B_n\cap B}/\Psi\rangle$, 
for $A\in\mathcal{B}$, is $\sigma$-additive.
Indeed, as $\nu(A)\in X$ and $(\chi_{B_n\cap B}/\Psi)\in L^\infty\subseteq X'$ 
we have, for each $A\in\mathcal{B}$, that
\begin{align*}
\langle\nu(A),\chi_{B_n\cap B}/\Psi\rangle &= 
\int_{-1}^1\nu(A)(x)\chi_{B_n\cap B}(x)/\Psi(x)\,dx
\\
&=
 \int_{-1}^1\mu(A)(x)\chi_{B_n\cap B}(x)\,dx
=\langle\mu(A),\chi_{B_n\cap B}\rangle,
\end{align*}
which implies the desired $\sigma$-additivity because 
$\mu$ is $\sigma$-additive as an $L^2$-valued vector measure
and $\chi_{B_n\cap B}\in L^2$.
Consequently, each $\mathbb{C}$-valued, $\sigma$-additive measure 
$A\mapsto\langle\nu(A),\chi_{B_n\cap B}/\Psi\rangle$ on $\mathcal{B}$, for
$n\in\mathbb{N}$, has bounded range. Recalling 
that the subset \eqref{c2} of $X^*$ is total for $X$,
the Dieudonn\'e-Grothendieck Theorem, \cite[Corollary I.3.3]{diestel-uhl},
implies that $\nu$ has bounded range in $X$. Hence, (b) is established.

(b)$\Rightarrow$(c). The \textit{semivariation} $\|\nu\|(\cdot)$ of the 
bounded, finitely additive, $X$-valued measure $\nu$ defined in \eqref{nu} satisfies both
$$
\|\nu\|(A)=\sup\big\{\|T(\chi_A fs)\|_X: s\in\text{sim }\mathcal{B},\; 
|s|\le1\big\},\quad A\in\mathcal{B},
$$
and
$$
\sup_{B\in \mathcal{B}, B\subseteq A}\|\nu(B)\|_X\le \|\nu\|(A)
\le 4 \sup_{B\in\mathcal{B}, B\subseteq A}\|\nu(B)\|_X,\quad A\in\mathcal{B},
$$
\cite[p.2 and Proposition I.1.11]{diestel-uhl}. Thus, for 
$s\in\text{sim }\mathcal{B}$ with $s\not=0$,
\begin{equation}\label{c3}
\|T(fs)\|_X\le \Big(4\sup_{A\in\mathcal{B}}\|T(f\chi_A)\|_X\Big)\cdot 
\sup_{|x|<1}|s(x)|<\infty
\end{equation}
because $|s|\le\sup_{|x|<1}|s(x)|$ pointwise on $(-1,1)$, \cite[p.6]{diestel-uhl}. To obtain 
(c) from \eqref{c3}, take any $h\in L^0$ with $|h|\le|f|$ a.e. Then $h=f\varphi$ for some 
$\varphi\in L^0$ with $|\varphi|\le1$ a.e. Select a sequence 
$\{s_n\}_{n=1}^\infty\subseteq \text{sim }\mathcal{B}$ such that 
$|s_n|\le |\varphi|$ on $(-1,1)$ for all $n\in\mathbb{N}$
and $s_n\to\varphi$ uniformly on $(-1,1)$ 
as $n\to\infty$. Then the sequence $\{T(fs_n)\}_{n=1}^\infty$ is 
Cauchy in $X$ as \eqref{c3} yields
$$
\|T(fs_j)-T(fs_k)\|_X \le \Big(4\sup_{A\in\mathcal{B}}\|T(f\chi_A)\|_X\Big)\cdot 
\sup_{|x|<1}|s_j(x)-s_k(x)|
$$
for all $j,k\in\mathbb{N}$. Accordingly, 
$\{T(fs_n)\}_{n=1}^\infty$  has a limit in $X$, say $g$.
Since the natural inclusion $X\subseteq L^{1,\infty}$ is continuous, we have
$\lim_nT(fs_n)=g$ in $L^{1,\infty}$. On the other hand, since
$\lim_nfs_n=f\varphi$ in $L^1$, Kolmogorov's Theorem gives   $\lim_n T(fs_n)=T(f\varphi)$ in 
$L^{1,\infty}$.  Thus, $T(h)=T(f\varphi)=g$ as elements of $L^0$. 
In particular, $T(h)\in X$ as $g\in  X$. So, (c) is established.

(c)$\Rightarrow$(d). Clearly (c)$\Rightarrow$(a) and we already know that 
(a)$\Rightarrow$(b). Thus, the previous 
arguments also imply the inequality
\begin{equation}\label{c4}
\sup_{|h|\le|f|}\|T(h)\|_X
\le 4 \sup_{A\in\mathcal{B}}\|T(f\chi_A)\|_X.
\end{equation}
To see this consider any $h\in L^0$ with $|h|\le|f|$ a.e. Select $\varphi$ and 
$\{s_n\}_{n=1}^\infty\subseteq \text{sim }\mathcal{B}$ as in the previous paragraph.
Then \eqref{c3} yields  
\begin{align*}
\|T(h)\|_X &= \lim_{n}\|T(fs_n)\|_X
\\
&\le 
\Big(4 \sup_{A\in\mathcal{B}}\|T(f\chi_A)\|_X\Big) 
\sup_{n\in\mathbb{N}} \sup_{|x|<1}|s_n(x)|
\\
& = 
\Big(4 \sup_{A\in\mathcal{B}}\|T(f\chi_A)\|_X\Big)  \sup_{|x|<1}|\varphi(x)|
\\
&\le 
4 \sup_{A\in\mathcal{B}}\|T(f\chi_A)\|_X  .
\end{align*}

(d)$\Rightarrow$(f)$\Rightarrow$(e) Clear.

(e)$\Rightarrow$(a) Fix $A\in\mathcal{B}$. Since 
$|\chi_A\pm\chi_{(-1,1)\setminus A}|=1$ it follows from
(e) that both
$$
T(f\chi_A)+T(f\chi_{(-1,1)\setminus A})=T(f(\chi_A+\chi_{(-1,1)\setminus A}))\in X
$$
and 
$$
T(f\chi_A)-T(f\chi_{(-1,1)\setminus A})=T(f(\chi_A-\chi_{(-1,1)\setminus A}))\in X.
$$
These two identities imply that $T(f\chi_A)\in X$.

(d)$\Rightarrow$(g). Fix $g\in X'$. Given $n\in\N$ define $A_n:=|f|^{-1}([0,n])$
and set $f_n:=f\chi_{A_n}\in L^\infty\subseteq X$. Since $|f_n|\uparrow |f|$ pointwise
on $(-1,1)$, the Monotone Convergence Theorem yields
\begin{equation}\label{dg}
\int_{-1}^1|f(x)|\cdot |(T_{X'}(g))(x)|\,dx=\lim_n
\int_{-1}^1|f_n(x)|\cdot |(T_{X'}(g))(x)|\,dx.
\end{equation}
Select $\theta_1, \theta_2\in L^\infty$ with $|\theta_1|=1$ and $|\theta_2|=1$ pointwise
such that $|f|=\theta_1 f$ and $|T_{X'}(g)|=\theta_2 T_{X'}(g)$ pointwise.
In particular, $|f_n|=\theta_1 f_n$ pointwise for all $n\in\N$. Then Parseval's
formula (cf. Proposition \ref{prop-2}(b)), H\"older's inequality and condition
(d) ensure, for every $n\in\N$, that
\begin{align*}
\int_{-1}^1|f_n(x)|\cdot |(T_{X'}(g))(x)|\,dx
&=
\int_{-1}^1\theta_1(x)\theta_2(x)f_n(x)(T_{X'}(g))(x)\,dx
\\ &=
- \int_{-1}^1(T_{X}(\theta_1\theta_2f_n))(x)g(x)\,dx
\\ & \le
\|T_{X}(\theta_1\theta_2f_n)\|_X \|g\|_{X'}
\\ & \le 
\sup_{|h|\le|f|}\|T(h)\|_X \|g\|_{X'} <\infty.
\end{align*}
This inequality and \eqref{dg} imply that (g) holds.

(g)$\Rightarrow$(a). Fix any $A\in\mathcal{B}$. 
Then $(f\chi_A)T_{X'}(g)\in L^1$ for every
$g\in X'$ by  assumption. Apply Proposition 
\ref{prop-2}(a) to $f\chi_A$ in place of $f$
to obtain that $gT(f\chi_A)\in L^1$ for all $g\in X'$. Accordingly, 
$T(f\chi_A)\in X''=X$, which establishes (a).

The equivalences (a)-(g) are thereby established.

Suppose now that any one of (a)-(g) is satisfied. The second 
inequality of \eqref{norms} is clear. For the left-hand inequality 
fix $A\in\mathcal{B}$. Then $T(f\chi_A)=
1/2(T(\theta_1 f)+T(\theta_2 f))$, where
$\theta_1=1$ and $\theta_2=\chi_A-\chi_{(-1,1)\setminus A}$ satisfy
$|\theta_1|=1$ and $|\theta_2|=1$. Accordingly,
$$
\|T(f\chi_A)\|_X\le 1/2(\|T(\theta_1 f)\|_X+\|T(\theta_2 f)\|_X)
\le 
\sup_{|\theta|=1}\|T(\theta f)\|_X.
$$
Finally, the last inequality in \eqref{norms} is precisely \eqref{c4}  above.
\end{proof}


  Another  consequence of Theorems \ref{theo-3} and \ref{theo-4} 
is that membership of a given r.i.\ space $X$ is completely 
determined by the finite Hilbert transform in $X$.

\begin{proposition}\label{cor-8}
Let $X$ be a  r.i.\ space    satisfying  either 
$1/2<\underline{\alpha}_X\le \overline{\alpha}_X<1$ or 
$0<\underline{\alpha}_X\le \overline{\alpha}_X<1/2$. 
Let $f\in L^1$. The following conditions are equivalent.
\begin{itemize}
\item[(a)] $f\in X$.
\item[(b)] $T(f\chi_A)\in X$ for every $A\in\mathcal{B}$.
\item[(c)] $T(f\theta)\in X$ for every  $\theta\in L^\infty$ with $|\theta|=1$ a.e.
\item[(d)] $T(h)\in X$ for every $h\in L^0$ with $|h|\le |f|$ a.e.
\end{itemize}
\end{proposition}

\begin{proof}
The three conditions (b), (c) and (d) are equivalent by Proposition \ref{prop-7}.

(a)$\Rightarrow$(b). Clear as $T\colon X\to X$ is bounded.

(b)$\Rightarrow$(a). By Proposition 4.1 we have $fT_{X'}(g)\in L^1$ for
every $g\in X'$, which we shall use to obtain (a).

Assume that $1/2<\underline{\alpha}_X\le\overline{\alpha}_X<1$, in which case
$0<\underline{\alpha}_{X'}\le\overline{\alpha}_{X'}<1/2$.
This enables us to apply Theorem \ref{theo-4}(c), 
with $X'$ in place of $X$, to the operator $T_{X'}$. So, for any $\psi\in X'$,
it follows that $\psi=T_{X'}(\widecheck{T}_{X'}(\psi))+c\mathbf{1}$ with
$c:=(1/\pi)\int_{-1}^1(\psi(x)/\sqrt{1-x^2})\,dx$. 
Define $g:=\widecheck{T}_{X'}(\psi)\in X'$. Then $fT_{X'}(\widecheck{T}_{X'}(\psi))\in L^1$
and hence, $f\psi-cf=fT_{X'}(\widecheck{T}_{X'}(\psi))$ belongs to $L^1$.
But, $cf\in L^1$ as $f\in L^1$ by assumption. So, $f\psi\in L^1$, 
from which it follows that   $f\in X''=X$ as $\psi\in X'$ is arbitrary. 
Thus (a) holds.

Consider  the remaining case when $0<\underline{\alpha}_X\le\overline{\alpha}_X<1/2$. 
Then $1/2<\underline{\alpha}_{X'}\le\overline{\alpha}_{X'}<1$. 
We  apply Theorem \ref{theo-3}(c) with $X'$ in place of $X$, to conclude that 
$T_{X'}\colon X'\to X'$ is surjective. So, 
given any $\psi\in X'$ there exists $g\in X'$ with $\psi=T_{X'}(g)$.  
It follows that $f\psi=fT_{X'}(g)\in L^1$. 
Since $\psi\in X'$ is arbitrary we may conclude that 
  $f\in X''=X$. Hence, (a) again holds.   
\end{proof}


Even though $T_X$ is not an isomorphism, 
Theorems \ref{theo-3} and \ref{theo-4} 
imply the impossibility of extending (continuously) the finite Hilbert transform 
$T_X\colon X\to X$ to any genuinely  larger domain space within $L^1$
while still maintaining its values in $X$;
see  Theorem \ref{theo-10} below. This is in contrast to the situation for the Fourier
transform operator acting in the spaces $L^p(\mathbb{T})$, $1<p<2$; see the Introduction.

We first require an important  technical construction. Define
\begin{equation*}\label{tx}
[T,X]:=\big\{f\in L^1: T(h)\in X, \;\forall |h|\le|f|\big\}.
\end{equation*}
If $f\in[T,X]$, then $f\in L^1$ and $T(h)\in X$ for 
every $h\in L^0$ with $|h|\le|f|$. Hence,  Proposition \ref{prop-7} implies 
that 
\begin{equation}\label{TX-norm}
\|f\|_{[T,X]}:=\sup_{|h|\le|f|} \|T(h)\|_X<\infty,\quad f\in[T,X].
\end{equation}
The properties of $[T,X]$ are established via a series of steps,
with the aim of showing that it is a B.f.s.

First, the functional $f\mapsto \|f\|_{[T,X]}$ is compatible with the 
lattice structure in the following sense:
if $f_1, f_2\in [T,X]$ satisfy $|f_1|\le|f_2|$, then 
$ \|f_1\|_{[T,X]}\le \|f_2\|_{[T,X]}$. This is because
 $\{h:|h|\le|f_1|\}\subseteq \{h:|h|\le|f_2|\}$. The same argument shows that $[T,X]$
is an ideal in $L^1$. In particular, $X\subseteq[T,X]$.

It is routine to verify that if $\alpha\in\mathbb{C}$ and $f\in[T,X]$, then 
$\alpha f\in[T,X]$ and $\|\alpha f\|_{[T,X]}=|\alpha |\cdot \|f\|_{[T,X]}$.

To verify the subadditivity of $\|\cdot\|_{[T,X]}$ 
we use the following Freudenthal type decomposition:  if
$h, f_1, f_2\in L^1$ with $|h|\le |f_1+f_2|$,
then  there exists $h_1, h_2$ such that $h=h_1+h_2$ and $|h_1|\le |f_1|, |h_2|\le |f_2|$;
this follows from \cite[Theorem 91.3]{zaanen} applied in $L^1$. Using this fact, given
$f_1,f_2\in [T,X]$, it follows that $f_1+f_2\in[T,X]$ and
\begin{align*}
\|f_1+f_2\|_{[T,X]} &= \sup\Big\{\|T(h)\|_X:|h|\le|f_1+f_2|\Big\}
\\
&=  \sup\Big\{\|T(h_1)+T(h_2)\|_X:|h|\le| f_1+f_2|, h=h_1+h_2, |h_i|\le|f_i|\Big\}
\\
&\le  \sup\Big\{\|T(h_1)\|_X:|h_1|\le|f_1|\Big\}
+  \sup\Big\{\|T(h_2)\|_X:|h_2|\le|f_2|\Big\}
\\
&=  \| f_1\|_{[T,X]}+ \| f_2\|_{[T,X]}.
\end{align*}
So, $[T,X]$ is a vector space and $\|\cdot\|_{[T,X]}$ is a lattice seminorm on $[T,X]$.

Let $\|f\|_{[T,X]}=0$.  Then $T(h)=0$ in $X$ for every $h\in L^0$ with 
$|h|\le|f|$. Suppose that
$f\not=0$. Then there exists $A\in\mathcal{B}$ with $|A|>0$ 
such that $f\chi_A\in L^\infty$ and
$f(x)\chi_A(x)\not=0$ for every $x\in A$. Choose two disjoint 
sets $A_1, A_2\in\mathcal{B}\cap A$
with $|A_j|>0$, $j=1,2$, and define $h_j:=f\chi_{A_j}$, $j=1,2$. 
Then $h_j\in L^\infty\subseteq X$ 
satisfies $|h_j|\le|f|$ and so $T_X(h_j)=T(h_j)=0$ for $j=1,2$. 
That is, $h_1,h_2\in \text{Ker}(T_X)$. Since $h_1, h_2$ are 
linearly independent elements in $X$,
it follows that $\text{dim}( \text{Ker}(T_X))\ge2$. But, this contradicts 
the fact that $T_X$ is either injective or its kernel is 1-dimensional; 
see the discussion after \eqref{0}. Hence, $f=0$.
So, we have shown that $[T,X]$ is a normed function space.

The following result is a Parseval type formula that will be
needed in the sequel.

\begin{lemma}\label{NEW-4.3}
Let $X$ be a r.i.\ space satisfying 
$0<\underline{\alpha}_X\le \overline{\alpha}_X<1$. Then
$$
\int_{-1}^1fT_{X'}(g)=-\int_{-1}^1gT(f),\quad f\in[T,X],\;\; g\in X'.
$$
\end{lemma}

\begin{proof}
Given $f\in[T,X]\subseteq L^1$, it follows from the definition of $[T,X]$
and Proposition \ref{prop-7} that $fT_{X'}(g)\in L^1$ for every $g\in X'$.
The desired formula in then immediate from Proposition \ref{prop-2}(a).
\end{proof}

\begin{lemma}\label{NEW-4.4}
Let $X$ be a r.i.\ space satisfying 
$0<\underline{\alpha}_X\le \overline{\alpha}_X<1$. Then
the normed function space $[T,X]$ is complete.
\end{lemma}

\begin{proof}
Let $f_n\in[T,X]$, for $n\in\N$, satisfy
$$
\sum_{n=1}^\infty \|f_n\|_{[T,X]}<\infty.
$$
This implies, for every choice of $h_n$ with $|h_n|\le|f_n|$, that
\begin{equation}\label{th}
\sum_{n=1}^\infty \|T(h_n)\|_X<\infty.
\end{equation}

(A) Let $h\in [T,X]\subseteq L^1$.  As $|h|\chi_{(-1,0)}\le|h|$ 
we have that $T(|h|\chi_{(-1,0)})\in X$.
If $0<t<1$, then 
$$
\left|T\left(|h|\chi_{(-1,0)}\right)(t)\right|
=\frac{1}{\pi}\int_{-1}^0\frac{|h(x)|}{|x-t|}dx \ge \frac{1}{2\pi}
\int_{-1}^0|h(x)|\,dx , 
$$
since for $-1<x<0$ and $0<t<1$ we have $|x-t|\le2$. Consequently, 
\begin{align*}
\|T\left(|h|\chi_{(-1,0)}\right)\|_X& \ge 
\|T\left(|h|\chi_{(-1,0)}\right)\chi_{(0,1)}\|_X \ge 
\left(\frac{1}{2\pi}\int_{-1}^0|h(x)|\,dx\right) \|\chi_{(0,1)}\|_X .
\end{align*}
In a similar way,  as $|h|\chi_{(0,1)}\le|h|$, 
we have that $T(|h|\chi_{(0,1)})\in X$. If $-1<t<0$, then 
$$
T\left(|h|\chi_{(0,1)}\right)(t)
=\frac{1}{\pi}\int_{0}^1\frac{|h(x)|}{x-t}dx \ge \frac{1}{2\pi}\int_{0}^1|h(x)|\,dx , 
$$
since for $-1<t<0$ and $0<x<1$ we have $0\le x-t\le2$. Consequently,
\begin{align*}
\|T\left(|h|\chi_{(0,1)}\right)\|_X &\ge 
\|T\left(|h|\chi_{(0,1)}\right)\chi_{(-1,0)}\|_X \ge 
\left(\frac{1}{2\pi}\int_{0}^1|h(x)|\,dx\right) \|\chi_{(-1,0)}\|_X.
\end{align*}

Applying \eqref{th} with $h_n:=|f_n|\chi_{(-1,0)}$ and 
$h_n:=|f_n|\chi_{(0,1)}$ it follows, from the previous bounds 
for $h=f_n$, that 
\begin{align*}
\sum_{n=1}^\infty \|f_n\|_{L^1}&=\sum_{n=1}^\infty 
\left( \int_{-1}^0|f_n(x)|\,dx + \int_{0}^1|f_n(x)|\,dx\right)
\\ &\le
\sum_{n=1}^\infty 
C\left(\|T\left(|f_n|\chi_{(-1,0)}\right)\|_X+ \|T\left(|f_n|\chi_{(0,1)}\right)\|_X\right)
\\ &\le
2C  \sum_{n=1}^\infty \|f_n\|_{[T,X]}<\infty,
\end{align*}
with $C:=(2\pi)/\varphi_X(1)$, since  $\|\chi_{(0,1)}\|_X=\|\chi_{(-1,0)}\|_X=\varphi_X(1)$.
Hence, we have
\begin{equation}\label{L1}
\sum_{n=1}^\infty f_n=:f\in L^1
\end{equation}
with absolute convergence in   $L^1$  and hence, also pointwise a.e.

(B) We now show that $f\in[T,X]$. Select $h\in L^0$ 
satisfying $|h|\le|f|$. We need to prove that $T(h)\in X$. 
To this end, let $\varphi\in L^0$ satisfy $|\varphi|\le1$
and $h=\varphi f$. Then
$$
h=\varphi f=\sum_{n=1}^\infty \varphi f_n,
\quad \mathrm{a.e.}
$$
The functions  $h_n:=\varphi f_n\in[T,X]$, for $n\in\mathbb{N}$,
satisfy
$$
\sum_{n=1}^\infty\|h_n\|_{[T,X]}\le \sum_{n=1}^\infty\|f_n\|_{[T,X]}<\infty
$$
due to the ideal property of $[T,X]$. We can apply the arguments in (A)
to deduce that the series $\sum_{n=1}^\infty h_n$ converges (absolutely) in
$L^1$ to $h$. Kolmogorov's Theorem yields that the series 
$\sum_{n=1}^\infty T(h_n)$ converges to $T(h)$ in $L^{1,\infty}$.

On the other hand, since the series $\sum_{n=1}^\infty T(h_n)$ converges
absolutely in $X$ (see \eqref{th}), it is convergent, say to 
$g=\sum_{n=1}^\infty T(h_n)$ in $X$ and hence, also in $L^{1,\infty}$. 
Accordingly, $T(h)=g$ and so $T(h)\in X$. 
This establishes that $f\in[T,X]$.

(C) It remains to show that $\sum_{n=1}^\infty f_n$ converges to $f$ in the topology of $[T,X]$, 
that is, $\|f-\sum_{n=1}^N f_n\|_{[T,X]}\to0$ as $N\to\infty$. Fix $N\in\mathbb{N}$. Let $h\in L^0$ satisfy
$$
|h|\le \left| f-\sum_{n=1}^N f_n\right| =\left|\sum_{n=N+1}^\infty f_n\right|\le \sum_{n=N+1}^\infty |f_n|.
$$
We can reproduce the argument used in (B) to deduce that
$$
h=\sum_{n=N+1}^\infty h_n,\quad |h_n|\le|f_n|,\quad n\ge N+1.
$$ 
Then
$$
\|T(h)\|_X\le\sum_{n=N+1}^\infty \|T(h_n)\|_X\le \sum_{n=N+1}^\infty \|f_n\|_{[T,X]}.
$$
That is, for each $N\in\mathbb{N}$, we have
$$
\|f-\sum_{n=1}^N f_n\|_{[T,X]} 
=\sup_{ |h|\le| f-\sum_{n=1}^N f_n|} \|T(h)\|_X\le \sum_{n=N+1}^\infty \|f_n\|_{[T,X]}\to0,
$$
which establishes the completeness of $[T,X]$.
\end{proof}

We will require an alternate description of the norm
$\|\cdot\|_{[T,X]}$ to that given in \eqref{TX-norm}, namely
\begin{equation}\label{III}
\|f\|_{[T,X]}=\sup_{\|g\|_{X'}\le1}\|fT_{X'}(g)\|_{L^1},\quad f\in[T,X].
\end{equation}
To verify this fix $f\in[T,X]$. Given $\varphi\in L^0$ with $|\varphi|\le1$,
the function $\varphi f\in[T,X]$ as $|\varphi f|\le|f|$. It follows from Lemma \ref{NEW-4.3}
(see also its proof) with $\varphi f$ in place of $f$, that
$\varphi fT_{X'}(g)\in L^1$ for all $g\in X'$ 
(in particular, also $fT_{X'}(g)\in L^1$) and
$$
\int_{-1}^1(\varphi f)T_{X'}(g)=-\int_{-1}^1gT(\varphi f),\quad g\in X'.
$$
Since $\{\varphi f: \varphi\in L^0, |\varphi|\le1\}=\{h\in L^0:|h|\le|f|\}$, the previous formula
yields \eqref{III} because \eqref{TX-norm} implies that
\begin{align*}
\|f\|_{[T,X]}&= \sup_{|\varphi|\le1}\|T(\varphi f)\|_X
= \sup_{|\varphi|\le1} \sup_{\|g\|_{X'}\le1}\Big|\int_{-1}^1 gT(\varphi f)\Big|
\\ &= 
\sup_{\|g\|_{X'}\le1}\sup_{|\varphi|\le1} \Big|\int_{-1}^1 (\varphi f)T_{X'}(g)\Big|
= 
\sup_{\|g\|_{X'}\le1}\|fT_{X'}(g)\|_{L^1}. 
\end{align*}

\begin{proposition}\label{prop 4.5}
Let $X$ be a   r.i.\ space  satisfying   
$0<\underline{\alpha}_X\le \overline{\alpha}_X<1$. 
Then $[T,X]$ is a B.f.s.
\end{proposition}

\begin{proof}
In view of Lemma \ref{NEW-4.4} it remains to establish that $[T,X]$ 
possesses the Fatou property.

Let $0\le f\in L^0$ and $\{f_n\}_{n=1}^\infty\subseteq [T,X]\subseteq L^1$
be a sequence such that $0\le f_n\le f_{n+1}\uparrow f$ pointwise a.e.\ with $\sup_n\|f_n\|_{[T,X]}<\infty$. In Step A of the proof of Lemma \ref{NEW-4.4} it was shown that
$$
\|h\|_{L^1}\le (4/\varphi_X(1))\|h\|_{[T,X]},\quad h\in [T,X],
$$
which ensures that also $\sup_n\|f_n\|_{L^1}<\infty$. Hence, via Fatou's lemma, 
$f\in L^1$. Moreover, the Monotone Convergence Theorem together with
\eqref{III} applied to $f_n\in[T,X]$ for each $n\in\N$ yields
\begin{align*}
\sup_{\|g\|_{X'}\le1}\int_{-1}^1|fT_{X'}(g)|
&=
\sup_{\|g\|_{X'}\le1}\sup_n\int_{-1}^1|f_nT_{X'}(g)|
\\ &= 
\sup_n\sup_{\|g\|_{X'}\le1}\int_{-1}^1|f_nT_{X'}(g)|
=
\sup_n\|f_n\|_{[T,X]}<\infty.
\end{align*}
In particular, $fT_{X'}(g)\in L^1$ for every $g\in X'$ with $f\in L^1$.
According to  (c)$\Leftrightarrow$(g) in Proposition \ref{prop-7}
we have $f\in[T,X]$ and, via \eqref{III} and the previous identity, that
$\|f\|_{[T,X]}=\sup_n\|f_n\|_{[T,X]}$. So, we have established that
$[T,X]$ has the Fatou property.
\end{proof}

The optimality property of the B.f.s.\ $[T,X]$ relative to $T_X$ 
can now be formulated.

\begin{theorem}\label{theorem 4.6}
Let $X$ be a   r.i.\ space  satisfying   
$0<\underline{\alpha}_X\le \overline{\alpha}_X<1$. 
Then $[T,X]$ is the largest B.f.s.\ containing $X$   
to which $T_X\colon X\to X$ has a continuous, linear, $X$-valued extension.
\end{theorem}

\begin{proof}
Let $Z\subseteq L^1$ be any B.f.s.\ with $X\subseteq Z$ such that
$T_X$ has a continuous, linear extension $T\colon Z\to X$. Fix $f\in Z$.
Then for each $h\in L^0$ with $|h|\le|f|$ we have $h\in Z$ and
$$
\|T(h)\|_X\le \|T\|_{op}\|h\|_Z\le\|T\|_{op} \|f\|_Z,
$$
where $\|T\|_{op}$ is the operator norm of $T\colon Z\to X$. Then
$f\in[T,X]$ and so the space $[T,X]$ contains $Z$ continuously.
Due to the boundedness of $T_X\colon X\to X$ we have that
$$
\|f\|_{[T,X]}=\sup_{|h|\le|f|}\|T(h)\|_X\le \|T_X\|_{op}\|f\|_X,\quad f\in X,
$$
and so $X\subseteq[T,X]$ continuously. By construction $T\colon[T,X]\to X$
and $T$ is continuous. Hence, $[T,X]$ is the
\textit{largest} B.f.s.\ containing $X$ to which $T_X\colon X\to X$ has a 
continuous, linear, $X$-valued extension.
\end{proof}


We can now prove the impossibility of extending $T_X\colon X\to X$.

\begin{theorem}\label{theo-10} 
Let $X$ be a  r.i.\ space   
satisfying  either 
$1/2<\underline{\alpha}_X\le \overline{\alpha}_X<1$ or 
$0<\underline{\alpha}_X\le \overline{\alpha}_X<1/2$. Then the finite Hilbert transform
$T_X\colon X\to X$ has no $X$-valued, continuous linear extension to any 
larger B.f.s.
\end{theorem}

\begin{proof} 
According to Theorem  \ref{theorem 4.6}, whenever 
$0<\underline{\alpha}_X\le \overline{\alpha}_X<1$, 
the space $[T,X]$ is the largest B.f.s.\
to which $T_X\colon X\to X$ can be continuously extended with $X\subseteq[T,X]$
continuously. So, it suffices  to prove that
 $[T,X]= X$. But, this corresponds precisely to the equivalence in 
Proposition \ref{cor-8} between the condition (a), i.e., $f\in X$, 
and the condition (d), i.e, $T(h)\in X$ for all $h\in L^0$ with $|h|\le|f|$,
which is the statement that  $f\in[T,X]$. 
\end{proof}


Recall that $T_X$ is not an isomorphism. Nevertheless, 
Theorems \ref{theo-3} and \ref{theo-4} 
yield  norms, in terms 
of the finite Hilbert transform, which are equivalent
to the given norm in the corresponding r.i.\ space.

\begin{corollary}\label{cor-10}
Let $X$ be a  r.i.\ space   
satisfying  either 
$1/2<\underline{\alpha}_X\le \overline{\alpha}_X<1$ or 
$0<\underline{\alpha}_X\le \overline{\alpha}_X<1/2$. 
Then there exists a constant $C_X>0$ such that
\begin{align*}
\frac{C_X}{4}\|f\|_X\le \sup_{A\in\mathcal{B}}\big\|T_X(\chi_A f)\big\|_X
& \le  \sup_{|\theta|=1}\big\|T_X(\theta f)\big\|_X
\\
& \le \sup_{|h|\le|f|}\big\|T_X(h)\big\|_X
\le  \|T_X\| \cdot \|f\|_X ,
\end{align*}
for every $f\in X$.
\end{corollary}

\begin{proof} 
The final inequality  is clear from
$$
\|T_X(h)\|_X\le \|T_X\|\cdot \|h\|_X\le \|T_X\|\cdot \|f\|_X
$$
for every $f\in X$ and every $h\in L^0$ with $|h|\le|f|$.

It was shown in the proof of Theorem \ref{theo-10}  that $[T,X]=X$. Hence,
there exists a constant $C_X>0$ such that
$$
C_X\|f\|_X\le \sup_{|h|\le|f|} \|T_X(h)\|_X,\quad f\in X.
$$

The   remaining inequalities now follow from \eqref{norms} which is applicable because 
if $f\in X$, then condition (c) in Proposition \ref{prop-7} is surely satisfied.
\end{proof}

\begin{remark} 
The notion of the optimal domain $[T,X]$ is meaningful  for a large family of operators acting on function spaces, 
as already commented in the Introduction. Amongst them, in a much simpler situation, are the positive operators. For a thorough study
of this topic see, for example, 
\cite{okada-ricker-sanchez} and the references therein.
\end{remark}


\section{The finite Hilbert transform on $L^2$}
\label{S5}


Theorems \ref{theo-3} and \ref{theo-4} are not applicable to $X=L^2$. 
Moreover,  $T_{L^2}$ is not
Fredholm and no inversion formula is available. Nevertheless,
it turns out that no extension of $T_{L^2}$ is possible. A new approach is 
needed to establish this. Trying to use the  results and techniques 
obtained for the cases $p\not=2$ in an attempt to study 
the possible extension of $T_{L^2}\colon L^2\to L^2$ 
is futile as shown  by the  following  consideration.
Let $X=L^p$ for  $1<p<2$ and set
$T_p:=T_{L^p}$. Since
$\underline{\alpha}_X= \overline{\alpha}_X=1/p\in(1/2,1)$, we are in the setting of 
Theorem \ref{theo-3}. The left-inverse of $T_p$ is the operator  
$\widehat{T}_p:=\widehat T_{L^p}$, defined by \eqref{T-hat}, that is, 
$$
\widehat{T}_p(f)(x):=\frac{-1}{\sqrt{1-x^2}}\, T_p(\sqrt{1-t^2}f(t))(x),
\quad \mathrm{a.e. }\; x\in (-1,1),
$$
which maps $L^p$ into $L^p$ and is an isomorphism onto its range. 
We estimate from below the operator 
norm of $\widehat{T}_p$. Since $T_p(\sqrt{1-t^2})(x)=-x$,   for $f:=\mathbf{1}$ we obtain
$$
\|\widehat{T}_p\|\ge \frac{\|x/\sqrt{1-x^2}\|_{L^p}}{\|\mathbf{1}\|_{L^p}}
=
\left(\frac12\int_{-1}^{1}\frac{|x|^p}{(1-x^2)^{p/2}}\,dx\right)^{1/p}
$$
which goes to $\infty$ as $p\to2^-$.

We denote
by $T_2$ the finite Hilbert transform  $T_{L^2}\colon L^2\to L^2$. The norm
$\|\cdot\|_{L^2}$ will simply be denoted by $\|\cdot\|_2$.

\begin{lemma}\label{lem-11}
For every  set $A\in\mathcal{B}$ we have
$$
\left\|T_2(\chi_A)\right\|_2 
\ge 
\left(\int_0^\infty\frac{4\lambda}{e^{\pi\lambda}+1}d\lambda\right)^{1/2} |A|^{1/2}.
$$
\end{lemma}

\begin{proof}
We rely on a consequence of the Stein-Weiss formula for
the distribution function of the Hilbert transform of a characteristic function,
due to Laeng, 
\cite[Theorem 1.2]{laeng}. Namely, for $A\subseteq\mathbb{R}$ with $|A|<\infty$, we have
$$
|\{x\in A:\left| H(\chi_A)(x))\right|>\lambda\}| = 
\frac{2|A|}{e^{\pi\lambda}+1},\quad \lambda>0.
$$
For $A\in\mathcal{B}$, it follows from properties of the distribution function for
$T_2(\chi_A)$ that 
\begin{align*}
\|T_2(\chi_A)\|^2_2
&=
\int_0^\infty 2\lambda \cdot |\{x\in (-1,1):\left| T_2(\chi_A)(x)\right|>\lambda\}|\,d\lambda
\\
&\ge
\int_0^\infty 2\lambda \cdot |\{x\in A:\left| H(\chi_A)(x)\right|>\lambda\}|\,d\lambda
\\
&= 
|A|\int_0^\infty  \frac{4\lambda}{e^{\pi\lambda}+1}\,d\lambda.
\end{align*}
\end{proof}


The approach we use for proving the impossibility of
extending $T_2$ is to show that $L^2$ coincides with
the B.f.s.\  $[T,L^2]$. For this, we need to compare
the norm in $L^2$ with the norm in $[T,L^2]$.

\begin{theorem}\label{theo-12}
For each  function $\phi\in\mathrm{sim }\;\mathcal{B}$ we have

\begin{equation*}\label{FHT-inq2}
\left(\int_0^\infty\frac{4\lambda}{e^{\pi\lambda}+1}d\lambda\right)^{1/2} 
\|\phi\|_2 \le \sup_{|\theta|=1}\big\|T_2(\theta \phi)\big\|_2 .
\end{equation*}
\end{theorem}

\begin{proof}
In order to prove the claim, fix any simple function 
$\phi=\sum_{n=1}^N a_n\chi_{A_n},$
with $a_n,\dots,a_N\in\C$ and pairwise disjoint sets 
$A_1,\dots, A_N\in\mathcal{B}$ with $N\in\N$.

Let $\tau$ denote the product measure on $\Lambda:=\{-1,1\}^N$ for the uniform probability 
on $\{-1,1\}$. Thus, given $\sigma\in\Lambda$ we have 
$\sigma=(\sigma_1,\dots,\sigma_N)$ with 
$\sigma_n=\pm1$ for $n=1,\dots, N$. Note that the coordinate projections 
$$
P_n:\sigma\in\Lambda\mapsto \sigma_n\in\{-1,1\}, \quad n=1,\dots,N,
$$
form an orthonormal set, i.e.,
\begin{equation}\label{orto}
\int_\Lambda P_jP_k\,d\tau=\int_\Lambda \sigma_j\sigma_k\,d\tau(\sigma)
=\delta_{j,k},\quad j,k=1,\dots,N.
\end{equation}

The function $F\colon\Lambda\to[0,\infty)$ defined by
$$
F(\sigma):=\left\|T_2\left(\sum_{n=1}^N \sigma_na_n\chi_{A_n}\right)\right\|_2,
\quad \sigma\in\Lambda,
$$
is bounded and measurable and so satisfies
\begin{equation}\label{2-oo}
\left\|F\right\|_{L^2(\tau)}\le \left\|F\right\|_{L^\infty(\tau)}.
\end{equation}
We now compute both of the norms in \eqref{2-oo} explicitly.

Given $\sigma=(\sigma_n)\in\Lambda$, the  measurable function defined on $(-1,1)$ by
$$
t\mapsto\theta_\sigma(t):=\chi_{(-1,1)\setminus(\cup_{n=1}^NA_n)}(t) + 
\sum_{n=1}^N\sigma_n\chi_{A_n}(t)
$$
satisfies $|\theta_\sigma|=1$ and  
$$
\theta_\sigma\phi=\sum_{n=1}^N\sigma_na_n\chi_{A_n}.
$$
Consequently, 
$$
T_2\big(\theta_\sigma\phi\big)=T_2\Big(\sum_{n=1}^N\sigma_na_n\chi_{A_n}\Big),
$$
from which   it is clear that
\begin{equation}\label{signs}
\left\|F\right\|_{L^\infty(\tau)}
=\sup_{\sigma\in\Lambda}
\bigg\|T_2\bigg(\sum_{n=1}^N \sigma_na_n\chi_{A_n}\bigg)\bigg\|_{2}
\le
\sup_{|\theta|=1}\big\|T_2(\theta \phi)\big\|_{2}.
\end{equation}

Set $\beta:=\big(\int_0^\infty
\frac{4\lambda}{e^{\pi\lambda}+1}d\lambda\big)^{1/2}$. 
By Fubini's theorem, \eqref{orto} and Lemma \ref{lem-11} it follows that
\begin{align*}
\left\|F\right\|^2_{L^2(\tau)}
&=
\int_\Lambda
\bigg\|T_2\bigg(\sum_{n=1}^N \sigma_na_n\chi_{A_n}\bigg)\bigg\|^2_2
\,d\tau(\sigma)
=
\int_\Lambda\int_{-1}^{1}
\bigg|\sum_{n=1}^N \sigma_na_nT_2(\chi_{A_n})(t)\bigg|^2dt
\,d\tau(\sigma)
\\
&=
\int_{-1}^{1}\int_\Lambda
\bigg|\sum_{n=1}^N \sigma_na_nT_2(\chi_{A_n})(t)\bigg|^2\,d\tau(\sigma)\,dt
=
\int_{-1}^{1}\sum_{n=1}^N \left|a_nT_2(\chi_{A_n})(t)\right|^2\,dt
\\
&=
\sum_{n=1}^N|a_n|^2\Big\|T_2(\chi_{A_n})\Big\|^2_2
\ge
\beta^2\sum_{n=1}^N |a_n|^2|A_n|
\\
& =  \beta^2\int_{-1}^{1}\bigg|\sum_{n=1}^N a_n\chi_{A_n}(t)\bigg|^2dt
\\
&  = \beta^2\|\phi\|^2_2.
\end{align*}
This inequality, together with \eqref{2-oo} and \eqref{signs}, yields
$$
\beta \|\phi\|_2 \le \sup_{|\theta|=1}\big\|T_2(\theta \phi)\big\|_2.
$$
Since the simple function $\phi$ is arbitrary, this establishes  the result.
\end{proof}


Theorem \ref{theo-12} implies
the impossibility of extending $T_2$.
Note that this does not follow from Theorem \ref{theo-10} since
$L^2$ does not satisfy the restriction on the Boyd indices.

\begin{theorem}\label{theo-14}
The finite Hilbert transform
$T_2\colon L^2\to L^2$ has no continuous, $L^2$-valued extension
to any genuinely larger B.f.s.
\end{theorem}

\begin{proof}
We follow the approach used for proving Theorem \ref{theo-10} to show that
$$
L^2=[T_2,L^2]:=\big\{f\in L^1: T_2(h)\in L^2,\;\forall |h|\le|f|\big\}.
$$
Note first note that
\begin{equation}\label{aa}
\beta \|\phi\|_2 \le \sup_{|\theta|=1}\big\|T_2(\theta \phi)\big\|_2 
\le \sup_{|h|\le|\phi|}\big\| T_2(h)\big\|_2,
\quad \phi\in\mathrm{sim }\;\mathcal{B} .
\end{equation}
The left-hand inequality is  Theorem \ref{theo-12}.
The right-hand inequality is clear from \eqref{norms}.

Let $f\in[T,L^2]$.
According to \eqref{aa}, 
for every $\phi\in\mathrm{sim }\;\mathcal{B}$ satisfying $|\phi|\le |f|$  
it follows that
$$
\beta\|\phi\|_2 \le
\sup_{|h|\le|f|}\big\|T_2(h)\big\|_2
=\|f\|_{[T,L^2]}.
$$
Taking the supremum with respect to all such  $\phi$ yields 
$\beta\|f\|_2 \le  \|f\|_{[T,L^2]}$.
This implies that $f\in L^2$. Consequently, 
$[T,L^2]=L^2$ with equivalent norms.
\end{proof}


A further  consequence of Theorem \ref{theo-12} 
leads to various equivalent norms, in terms 
of the operator $T_2$, to the standard norm $\|\cdot\|_2$ in $L^2$.
As before, note that this does not follow from Corollary \ref{cor-10} since
$L^2$ does not satisfy the restriction on the Boyd indices.
Recall that $\beta:=\big(\int_0^\infty
\frac{4\lambda}{e^{\pi\lambda}+1}d\lambda\big)^{1/2}$.

\begin{corollary}\label{cor-14}
For every $f\in L^2$, we have 
$$
\frac{\beta}{4}\|f\|_2
\le  \sup_{A\in\mathcal{B}}\big\|T_2(\chi_A f)\big\|_2
\le \sup_{|\theta|=1}\big\|T_2(\theta f)\big\|_2
\le \sup_{|h|\le|f|}\big\|T_2(h)\big\|_2
\le \|f\|_2 .
$$
\end{corollary}

\begin{proof}
The last inequality follows (since $\|\cdot\|_2$ is a lattice norm 
and $\|T_2\|=1$, \cite{mclean-elliot}) via
$$
\|T_2(h)\|_2\le \|T_2\|\cdot\|h\|_2\le \|f\|_2,\quad |h|\le |f| .
$$

If $f\in L^2$, then surely (c) of  Proposition \ref{prop-7} is satisfied
with $X=L^2$. Hence the second and third inequalities follow from \eqref{norms}.

Finally, in order to prove the first inequality, 
we begin by establishing, for $h,f\in L^2$ satisfying $|h|\le|f|$, that 
\begin{equation}\label{nueva}
\sup_{|\theta|=1}\big\|T(\theta h)\big\|_2
\le
\sup_{|\tilde\theta|=1}\big\|T(\tilde\theta f)\big\|_2 .
\end{equation}
Fix $\theta$ with $|\theta|=1$. Then, via 
Parseval's formula, 
for some function $\tilde\theta_{f,g}$ satisfying $|\tilde\theta_{f,g}|=1$, we have  
\begin{align*}
\big\|T_2(\theta h)\big\|_2
&= 
\sup_{\|g\|_2\le1}\left|\int_{-1}^1T_2(\theta h)(t)\cdot g(t)\,dt\right|
=
\sup_{\|g\|_2\le1}\left|\int_{-1}^1\theta(t) h(t)\cdot T_2(g)(t)\,dt\right|
\\
&\le 
\sup_{\|g\|_2\le1}\int_{-1}^1|h(t)|\cdot |T_2(g)(t)|\,dt
\le
\sup_{\|g\|_2\le1}\int_{-1}^1|f(t)| \cdot|T_2(g)(t)|\,dt
\\
&= 
\sup_{\|g\|_2\le1}\int_{-1}^1f(t) \tilde\theta_{f,g} (t)T_2(g)(t)\,dt
\le
\sup_{\|g\|_2\le1}\left|\int_{-1}^1T_2(f \tilde\theta_{f,g})(t) g(t)\,dt\right|
\\
&\le 
\sup_{\|g\|_2\le1}\|T_2(f \tilde\theta_{f,g})\|_2 \|g\||_2
\\
&\le 
\sup_{|\tilde\theta|=1}\|T_2(f \tilde\theta)\|_2 .
\end{align*}
Accordingly, \eqref{nueva} holds.

Fix $f\in L^2$. Then Theorem \ref{theo-12}, together 
with  \eqref{norms} and \eqref{nueva} gives, for 
$\phi\in\mathrm{sim}\;\mathcal{B}$  satisfying $|\phi|\le |f|$,
that
$$
\beta\|\phi\|_2 
\le  
\sup_{|\theta|=1}\big\|T_2(\theta \phi)\big\|_2
\le
\sup_{|\theta|=1}\big\|T_2(\theta f)\big\|_2
\le
4\sup_{A\in\mathcal{B}}\big\|T_2(f\chi_A)\big\|_2.
$$
Taking the supremum with respect to all such simple functions $\phi$ , we arrive at
$$
\beta\|f\|_2 \le  
4\sup_{A\in\mathcal{B}}\big\|T_2(f\chi_A)\big\|_2.
$$
\end{proof}


From Corollary \ref{cor-14} we can deduce conditions,  in terms 
of the finite Hilbert transform, for membership  of $L^2$.

\begin{corollary}\label{cor-15}
Given $f\in L^1$ the following conditions are equivalent.
\begin{itemize}
\item[(a)] $f\in L^2$.
\item[(b)] $T(f\chi_A)\in L^2$ for every $A\in\mathcal{B}$.
\item[(c)] $T(f\theta)\in L^2$ for every  $\theta\in L^\infty$ with $|\theta|=1$ a.e.
\item[(d)] $T(h)\in L^2$ for every $h\in L^0$ with $|h|\le |f|$ a.e.
\end{itemize}
\end{corollary}

\begin{proof}
(b)$\Leftrightarrow$(c)$\Leftrightarrow$(d) follow from Proposition \ref{prop-7} with $X=L^2$.

(a)$\Rightarrow$(b) Clear as $T_2\colon L^2\to L^2$ is bounded.

(b)$\Rightarrow$(a) For $X=L^2$  it follows that 
condition (b) of Proposition \ref{prop-7} holds, that is, 
$\gamma:=\sup_{A\in\mathcal{B}}\|T(f\chi_A)\|_2<\infty$.
For each $n\in\mathbb{N}$ define $A_n:=|f|^{-1}([0,n])$ and $f_n:=f\chi_{A_n}$. Then
$$
\|T(f_n\chi_A)\|_2=\|T(f\chi_{A\cap A_n})\|_2\le 
\gamma,\quad A\in\mathcal{B}, n\in\mathbb{N},
$$
which implies, via Corollary \ref{cor-14}, that
$$
\|f_n\|_2\le \frac{4\gamma}{\beta},\quad n\in\mathbb{N}.
$$
Since $|f_n|^2\uparrow|f|^2$ pointwise a.e.\ 
on $(-1,1)$, from the Monotone Convergence Theorem
it follows that $f\in L^2$. This is condition (a).
\end{proof}


\begin{remark}
As commented in the Introduction 
the operator $T_2\colon L^2\to L^2$ is injective and has proper dense range.
A detailed study of its range is carried out in Sections 3 and 4 of \cite{okada-elliot}. 
Let us highlight a somewhat unexpected result given there. Namely,
for every $-1<a<1$, 
each function $f_a(x):=\chi_{(a,1)}(x)/\sqrt{1-x^2}$, for $x\in(-1,1)$, which belongs to $L^1$,
satisfies $T(f_a)\in L^2$ and
$$
\left\|T(f_a)\right\|_2=
\left\|
T\left(\frac{\chi_{(a,1)}}{\sqrt{1-x^2}}\right)
\right\|_2
=\frac1\pi \big(7\zeta(3)\big)^{1/2},
$$
\cite[Lemma 4.3 and Note 4.4]{okada-elliot}. Observe that $f_a\not\in L^2$ for every 
$-1<a<1$. On the other hand, if $X$ is a  r.i.\ space satisfying 
$1/2<\underline{\alpha}_X\le\overline{\alpha}_X<1$, then $K=\{f_a: -1<a<1\}
\subseteq L^{2,\infty}\subseteq X$. Moreover, for every sequence $a_n\uparrow1^-$
the sequence $\{f_{a_n}\}_{n=1}^\infty$ satisfies $0\le f_{a_n}\downarrow0$ pointwise.
By the absolute continuity of the norm $\|\cdot\|_X$ it follows that $\lim_nT_X(f_{a_n})=0$ in $X$.
\end{remark}

\begin{remark}\label{rem-final} 
For r.i.\ spaces $X$ satisfying the conditions   of Theorem \ref{theo-10}, namely 
\begin{equation}\label{eq-5.6}
0<\underline{\alpha}_X\le \overline{\alpha}_X<1/2\quad \text{or}\quad 
1/2<\underline{\alpha}_X\le \overline{\alpha}_X<1,
\end{equation}
we know that the finite Hilbert transform $T_X\colon X\to X$ cannot be extended
to a larger B.f.s. The proof is based on arguments from  Fredholm operator theory, a deep factorization result
of Talagrand on $L^0$-valued measures and on the construction of the largest domain space $[T,X]$.
For  r.i.\ spaces  $X$ with $0<\underline{\alpha}_X\le \overline{\alpha}_X<1$ not satisfying the 
conditions  \eqref{eq-5.6} it is unknown in general when $T_X$ is Fredholm and when not (for $X=L^2$ it is known that $T_X$ is not Fredholm). So, the arguments used to prove Theorem \ref{theo-10} may apply to some further cases but surely not to all.
The proof given in Theorem \ref{theo-14}  for  $X=L^2$ relies heavily on properties of the $L^2$-setting. 
Thus, it is difficult to extend to other spaces. The possibility of a related proof,  
at least for the spaces $L^{2,q}$ with $1\le q\le \infty$ and $q\not=2$, would 
require carefully looking at  the  ``measure of level sets''. 
Many technical difficulties would be expected to  arise  in such an attempt and still not all cases would be covered. Nevertheless, the class of r.i.\ spaces $X$ having the property \eqref{eq-5.6}, together with $X=L^2$, is rather large and suggests that $[T,X]=X$ should hold for all r.i.\ spaces satisfying $0<\underline{\alpha}_X\le \overline{\alpha}_X<1$.
\end{remark}




\begin{thebibliography}{99}


\bibitem{astala-etal}
K. Astala, L. P\"aiv\"arinta, E. Saksman, 
\emph{The finite Hilbert transform in weighted spaces},
Proc. Roy. Soc. Edinburgh Sect. A \textbf{126} (1996),  1157--1167.

\bibitem{bennett-sharpley} 
C. Bennett, R. Sharpley, 
\emph{Interpolation of Operators}, Academic Press,  Boston, 1988.

\bibitem{butzer-nessel}
P. L. Butzer, R. J. Nessel,
\emph{Fourier Analysis and Approximation.
Vol.\ 1: One-dimensional Theory}, Academic Press, New York-London, 1971. 

\bibitem{cheng-rott}
H. K.Cheng, N. Rott, 
\emph{Generalizations of the inversion formula of thin airfoil theory},
J. Rational Mech. Anal. \textbf{3} (1954), 357--382. 

\bibitem{curbera-ricker-nach}
G. P. Curbera, W. J. Ricker, \emph{Optimal domains for kernel operators via interpolation},
Math. Nachr. \textbf{244} (2002),  47--63.

\bibitem{curbera-ricker-sm}
G. P. Curbera, W. J. Ricker, \emph{Optimal domains for 
the kernel operator associated with Sobolev's inequality},
Studia Math. \textbf{158} (2003),  131--152 and  \textbf{170} (2005), 217--218.

\bibitem{curbera-ricker-tams}
G. P. Curbera, W. J. Ricker, \emph{Compactness properties of
Sobolev imbeddings for rearrangement invariant norms},
Trans. Amer Math. Soc. \textbf{359} (2007),  1471--1484.

\bibitem{curbera-ricker-ind}
G. P. Curbera, W. J. Ricker, 
\emph{Can optimal rearrangement invariant Sobolev imbeddings 
be further extended?},
Indiana Univ. Math. J. \textbf{56} (2007),  1479--1497.

\bibitem{cwikel-pustylnik}
M. Cwikel, E. Pustylnik, \emph{Sobolev type embeddings in the limiting case}, J. Fourier
Anal. Appl. \textbf{4} (1998), 433-446.

\bibitem{delgado-soria}
O. Delgado, J. Soria, 
\emph{Optimal domain for the Hardy operator},
J. Funct. Anal.  \textbf{244} (2007),   119--133.


\bibitem{diestel-uhl}
J. Diestel, J. J. Uhl, Jr., \emph{Vector Measures}, Math. Surveys
15, Amer. Math. Soc., Providence, R.I., 1977.

\bibitem{edmunds-kerman-pick}
D. Edmunds, R. Kerman, L. Pick,
\emph{Optimal Sobolev imbeddings involving rearrangement invariant
quasinorms}, J. Funct. Anal. \textbf{170} (2000), 307--355.

\bibitem{edwards}
R. E. Edwards, \emph{Fourier Series, A Modern Introduction, Vol. 2}, Holt, Rinehart and
Winston, New York-Montreal, 1967.

\bibitem{edwards-gaudry}
R. E. Edwards, G. I. Gaudry, 
\emph{Littlewood-Paley and Multiplier Theory},
Springer-Verlag, Berlin, 1977. 

\bibitem{gohberg-krupnik}
I. Gohberg, N. Krupnik, 
\emph{One-Dimensional Linear Singular Integral Operators}, 
Operator Theory Advances and Applications \textbf{53}, Birkh\"auser,  Berlin, 1992.

\bibitem{katsevich-tovbis}
A. Katsevich, A. Tovbis, \emph{Finite Hilbert transform with 
incomplete data: null-space and singular values},  
Inverse Problems \textbf{28} (2012), 105006, 28 pp.

\bibitem{king}
F. W. King, \emph{Hilbert Transforms
Vol. I}, Cambridge University Press, Cambridge New York, 2009.

\bibitem{laeng}
E. Laeng, \emph{On the $L^p$ norms of the Hilbert transform of a characteristic function},
J. Funct. Anal. \textbf{262} (2012), 4534--4539.

\bibitem{lewis}
D. R. Lewis, \emph{Integration with respect to vector measures},
Pacific J. Math. \textbf{33} (1970), 157--165.

\bibitem{lindenstrauss-tzafriri}
J. Lindenstrauss,  L. Tzafriri, \emph{Classical Banach Spaces
Vol. II}, Springer-Ver\-lag, Berlin, 1979.

\bibitem{love}
E. R. Love, \emph{Repeated singular integrals},
J. London Math. Soc. (2) \textbf{15} (1977),  99-102.

\bibitem{mclean-elliot}
W. McLean, D. Elliot, \emph{On the p-norm of the truncated Hilbert transform},
Bull. Austral. Math. Soc.  \textbf{38} (1988),  413-420.

\bibitem{mockenhaupt-okada-ricker}
G. Mockenhaupt, S. Okada, W. J. Ricker, \emph{Optimal extension of
 Fourier multiplier operators in $L^p(G)$},
Integral Equations Operator Theory \textbf{68} (2010), 573--599.

\bibitem{mockenhaupt-ricker}
G. Mockenhaupt, W. J. Ricker, \emph{Optimal extension of the Hausdorff-Young inequality},
J. reine angew. Math. \textbf{620} (2008), 195--211.

\bibitem{okada-elliot}
S. Okada, D. Elliot, \emph{The finite Hilbert transform in $\mathcal{L}^2$},
Math. Nachr. \textbf{153} (1991),  43--56.

\bibitem{okada-ricker-sanchez}
S. Okada, W. J. Ricker, E. A. S\'anchez-P\'erez, 
\emph{Optimal Domain and Integral Extension of Operators: Acting in Function Spaces}, 
Operator Theory Advances and Applications \textbf{180}, Birkh\"auser,  Berlin, 2008.

\bibitem{reissner}
E. Reissner, \emph{Boundary value problems in 
aerodynamics of lifting surfaces in non-uniform motion}, 
Bull. Amer. Math. Soc. \textbf{55} (1949), 825--850. 

\bibitem{sidky-etal}
E. Y. Sidky, Xiaochuan Pan, \emph{Recovering a compactly 
supported function from knowledge of its Hilbert transform on a finite interval},
IEEE Signal Processing Letters \textbf{12} (2005), 97--100.

\bibitem{sohngen}  
H. S\"ohngen, \emph{Zur Theorie der endlichen Hilbert-Transformation}, 
Math. Z. \textbf{60} (1954),  31--51.

\bibitem{talagrand} 
M. Talagrand, \emph{Les mesures vectorielles \`{a}
valeurs dans $L^0$ sont born\'ees}, Ann. Sci. \'Ecole Norm. Sup. (4)
\textbf{14} (1981), 445--452.

\bibitem{tricomi-1}
F. G. Tricomi,   \emph{On the finite Hilbert transform}, 
Quart. J. Math. \textbf{2} (1951),  199--211.

\bibitem{tricomi}
F. G. Tricomi, \emph{Integral Equations}, 
Interscience, New York, 1957.

\bibitem{zaanen}
A. C. Zaanen, 
\emph{Riesz Spaces II}, 
North Holland, Amsterdam, 1983.



\end{thebibliography}
\end{document}